\documentclass[12pt]{article}

\usepackage{amsmath}
\usepackage{amssymb}
\usepackage{latexsym}
\usepackage[dvips]{graphicx}
\usepackage{psfrag}
\usepackage{hyperref}
\usepackage{multirow}

\newtheorem{theorem}{Theorem}[section]
\newtheorem{lemma}{Lemma}[section]
\newtheorem{corollary}{Corollary}[section]
\newtheorem{prop}{Proposition}[section]
\newtheorem{definition}{Definition}[section]
\newtheorem{remark}{Remark}
\newtheorem{example}{Example}

\newenvironment{proof}{\noindent{\textsc{Proof.}}}
{$\hfill\Box$\vspace{0.1 cm}\\}

\newcommand{\R}{\mathbb R}

\newcommand{\N}{{\mathbb N}}
\newcommand{\Z}{{\mathbb Z}}
\newcommand{\Q}{{\mathbb Q}}

\newcommand{\C}{{\mathcal C}}
\newcommand{\A}{{\mathcal A}}
\newcommand{\B}{{\mathcal B}}
\newcommand{\PP}{\mathcal{P}}

\newcommand{\WW}{\mathcal{W}}
\newcommand{\DD}{\mathcal{D}}

\begin{document}


\title{Measure differential equations}

\author{Benedetto Piccoli}

\maketitle

\begin{abstract}
A new type of differential equations for probability measures on Euclidean spaces,
called Measure Differential Equations (briefly MDEs), is introduced.
MDEs correspond to Probability Vector Fields, which map measures on an Euclidean
space to measures on its tangent bundle. Solutions are intended in weak sense and 
existence, uniqueness and continuous dependence results are proved under suitable conditions.
The latter are expressed in terms of the Wasserstein metric on the base and fiber of the tangent bundle.\\
MDEs represent a natural measure-theoretic generalization of 
Ordinary Differential Equations via a monoid morphism mapping sums of vector fields
to fiber convolution of the corresponding Probability Vector Fields.
Various examples, including finite-speed diffusion and concentration, are shown,
together with relationships to Partial Differential Equations.
Finally, MDEs are also natural mean-field limits of multi-particle systems,
with convergence results extending the classical Dubroshin approach.
\end{abstract}

\section{Introduction}
The evolution of many physical and biological systems can be modeled by ordinary or
partial differential equations. To include a representation of
uncertainties, the state of the system can be modeled by
a probability distribution or a random variable rather than a point 
of an Euclidean space (or a manifold.)
Stochastic differential equations (SDEs) \cite{Oksendal} 
offer a well-developed and successful tool
to describe the evolution of random variables. 

We define a new type of differential equations for probability measures.
The point of view is that of optimal transport, thus we endow
the space of probability measures (on an Euclidean space $\R^n$)
with the Wasserstein metric.
The latter is defined in terms of solutions to the optimal transport problem,
first proposed by Monge in 1781 and then extended by Kantorovich in 1942, see \cite{Villani} for more complete historical perspectives. 
We first introduce the concept of Probability Vector Field (briefly PVF),
which is a map assigning to every probability measure $\mu$ on $\R^n$
a probability measure $V[\mu]$ on $T\R^n$ (the tangent bundle), 
whose marginal on the base is  $\mu$ itself. 
In simple words, the fiber values of $V[\mu]$ provide the velocities along which the mass of $\mu$ is spread.
Given a PVF $V$, the corresponding Measure Differential Equation (briefly MDE) reads
$\dot{\mu}=V[\mu]$ and a solution is defined in the usual weak sense.
If $V$ is sublinear (for the size of measures' support)
and continuous, w.r.t. the Wasserstein metrics on $\R^n$ and $T\R^n$, then
we obtain a solution using approximation and compactness.
More precisely, by discretizing in space, time and velocities we construct
approximate solutions consisting of finite sums of Dirac deltas 
moving on a lattice of $\R^n$, called Lattice Approximate Solutions  (briefly LASs.) 
LASs can be seen as  generalizations of probabilistic Cellular Automata, 
defined using $V$.

To address continuous dependence from initial data, 
it is not enough to ask for Lipschitz continuity of $V$
for the Wasserstein metrics. This is due to the fact that the fiber marginal
of $V[\mu]$ has a meaning of an infinitesimal displacement,
opposed to the base marginal. Therefore we introduce a different quantity,
which compute the Wasserstein distance over the fiber restricted
to transference plans which are optimal over the base, see (\ref{eq:WW}).
This allows to obtain the existence of a Lipschitz semigroup of solutions,
obtained as limit of LASs, from Lipschitz-type assumptions.
Weak solutions to Cauchy Problems for MDEs are not expected to be unique, thus we address the question
of uniqueness at the level of semigroup. For this purpose, we introduce the concept
of Dirac germ, which consists of small-time evolution for finite sums
of Dirac deltas. Then we show uniqueness of a Lipschitz semigroup,
compatible with a given Dirac germ. Therefore
uniqueness questions can be addressed by looking for unique limits to LASs with finite sums of Dirac deltas as initial data. 

Then we explore various connections of MDEs with classical approaches.
First, we show that MDEs represent a natural measure-theoretic generalization
of Ordinary Differential Equations (briefly ODEs).
A MDE is naturally associated to an ODE by moving masses along the ODE solutions.
Lipschitz continuity of the ODE implies existence of a Lipschitz semigroup for the corresponding MDE
(which is the only one compatible with ODE solutions.)
The correspondence ODEs-MDEs defines a map, which is
a monoid morphism between the space of vector fields, endowed with the usual sum,
and the space of PVFs, endowed with a fiber-convolution operation. Moreover, the map
sends the multiplication by a scalar to the natural counterpart of scalar
multiplication over the fiber, see Proposition \ref{prop:Vec-PVF}.\\
MDEs can model both diffusion and concentration phenomena.
We first show that a PVF $V$, which is constant on the fiber component,
gives rise to a simple translation (because of the Law of Large Numbers.)
On the other side, it is possible to define PVFs, which depend on the
global properties of the measures, providing finite speed diffusion.
For MDEs representing concentration,
uniqueness is obtained by one-sided Lipschitz-type
conditions, mimicking the one-sided Lipschitz conditions for ODEs.
Moreover, MDEs extend the theory of conservation laws with discontinuous fluxes.\\
Finally, kinetic models are considered. 
Dobrushin's approach (\cite{Dobrushin}) is recovered as a special case 
of MDEs in the following sense. Given a multi-particle system, whose dynamic is given
by ODEs, one can define a corresponding MDE under appropriate conditions 
(e.g. indistinguishibility of particles and uniform Lipschitz estimates.)
Moreover, the MDE enjoys well-posedness properties and compatibility with the 
empirical probability distributions defined by the multi-particle system.

The paper is organized as follows. In Section \ref{sec:BD} we define PVFs,
MDEs and solutions to MDEs. Then, in Section \ref{sec:Cauchy}, we prove
existence of solutions to Cauchy Problems for MDEs under continuity assumption,
and, in Section \ref{sec:Lip}, the existence of a Lipschitz semigroup of solutions
under appropriate Lipschitz-type assumptions. Uniqueness of Lipschitz semigroups
is addressed in Section \ref{sec:Dirac} using the concept of Dirac germ 
(Definition \ref{def:Dir-germ}). The relationship of MDEs with ODEs is explored
in Section \ref{sec:ODEs}, while examples of finite-speed diffusion
and concentration phenomena are given in Section \ref{sec:Ex}.
Finally, results for mean-field limits of multi-particle systems, 
seen as special cases of MDEs, are provided in Section \ref{sec:kin}.

\section{Basic definitions}\label{sec:BD}
For simplicity we restrict to $\R^n$, but a local theory can be easily developed 
for manifolds admitting a partition of unity. For every $R>0$, $B(0,R)$ indicates
the ball of radius $R$ centered at the origin,
$T\R^n$ the tangent bundle of $\R^n$,
and $\pi_1:T\R^n\to \R^n$ the projection to the base $\R^n$,
i.e. $\pi_1(x,v)=x$.
We also define $\pi_{13}:(T\R^n)^2\mapsto (\R^n)^2$ by
$\pi_{13}(x,v,y,w)=(x,y)$ (i.e. the projection on the bases for both components).
For every $A\subset \R^n$, $\chi_A$ indicates the characteristic function of the set $A$
and ${\C}^\infty_c(\R^n)$ indicates the space of smooth functions with compact support.

Given $(X,d)$ Polish space (complete separable metric space) 
we indicate by $\PP(X)$ the set of probability measures
on $X$, i.e. positive Borel measures with total mass equal to one. 
Given $\mu\in\PP(X)$ we indicate by $Supp(\mu)$ its support and we define
$\PP_c(X)$ to be the set of probability measures with compact support.
Given $(X_1,d_1)$, $(X_2,d_2)$ Polish spaces, $\mu\in\PP(X_1)$ 
and $\phi:X_1\to X_2$ measurable, we define the push forward
$\phi\#\mu\in\PP(X_2)$ by $\phi\#\mu(A) = \mu ( \phi^{-1}(A) ) =
\mu(\{x\in X_1:\phi(x)\in A\})$.
Given $\mu\in\PP(X_1)$ and $\nu_x\in\PP(X_2)$, $x\in X_1$,
we define $\mu\otimes \nu_x$ by 
$\int_{X_1\times X_2} \phi(x,v) \ d(\mu\otimes \nu_x) =
\int_{X_1} \int_{X_2} \phi(x,v) d\nu_x(v)\ d\mu(x)$.

\begin{definition}
A Probability Vector Field (briefly PVF)  on $\PP(\R^n)$ is a map 
$V: \PP(\R^n)\to\PP(T\R^n)$ such that $\pi_1\# V[\mu]=\mu$.
\end{definition}
Given a PVF $V$, we define the corresponding Measure Differential Equation (MDE) by:
\begin{equation}\label{eq:MDE}
\dot\mu=V[\mu].
\end{equation}
In simple words, $V[\mu]$ restricted to $T_xX$
indicates the directions towards which the mass of $\mu$ at $x$ is spread.
For every $\mu_0\in\PP(\R^n)$ we define the Cauchy problem:
\begin{equation}\label{eq:MDE-Cauchy}
\dot\mu=V[\mu],\qquad \mu(0)=\mu_0.
\end{equation}
A solution to (\ref{eq:MDE-Cauchy}) in weak sense is defined as follows:
\begin{definition}\label{def:sol-MDE}
A solution to (\ref{eq:MDE-Cauchy}) is a map $\mu:[0,T]\to \PP(\R^n)$ 
such that $\mu(0)=\mu_0$ and the following holds.
For every $f\in{\C}^\infty_c(\R^n)$,  the integral
$\int_{TR^n} (\nabla f(x)\cdot v)\ dV[\mu(s)](x,v)$ is defined for almost every $s$,
the map $s\to \int_{TR^n} (\nabla f(x)\cdot v)\ dV[\mu(s)](x,v)$ belongs to $L^1(\R)$,
and the map $t\to \int f\, d\mu(t)$ is absolutely
continuous and for almost every $t\in [0,T]$ it satisfies:
\begin{equation}\label{eq:sol-MDE}
\frac{d}{dt}\int_{R^n} f(x)\,d\mu(t)(x) = 
\int_{TR^n} (\nabla f(x)\cdot v)\ dV[\mu(s)](x,v).
\end{equation}
\end{definition}
Alternatively we may ask the following condition to hold 
for every $f\in{\C}^\infty_c(\R^n)$:
\begin{equation}\label{eq:MDE-int}
\int_{R^n} f(x)\,d\mu(t)(x) = \int_{R^n} f(x)\,d\mu_0(x) \ +
\int_0^t  \int_{TR^n} (\nabla f(x)\cdot v)\ dV[\mu(s)](x,v)\ ds.
\end{equation}
We also have the following equivalent formulation as distributional solution on $[0,T]\times\R^n$:
\begin{lemma}\label{lem:distr-sol}
Consider a map $\mu:[0,T]\to \PP(\R^n)$, with $\mu(0)=\mu_0$ and such that
for every $f\in{\C}^\infty_c(\R^n)$,  the integral
$\int_{TR^n} (\nabla f(x)\cdot v)\ dV[\mu(s)](x,v)$ is defined for almost every $s$ and
the map $s\to \int_{TR^n} (\nabla f(x)\cdot v)\ dV[\mu(s)](x,v)$ belongs to $L^1(\R)$.
Then $\mu$ is a solution to (\ref{eq:MDE-Cauchy})
if and only if for every $g\in{\C}^\infty_c([0,T]\times \R^n)$ it holds
\begin{eqnarray}\label{eq:MDE-distr}
& \int_{R^n} g(T,x)\,d\mu(T)(x) - \int_{R^n} g(0,x)\,d\mu_0(x) \ =\nonumber\\
& \int_0^T  \int_{R^n} \partial_s g(s,x)\ d\mu(s)\,ds+
\int_0^T  \int_{TR^n} (\nabla_x g(s,x)\cdot v)\ dV[\mu(s)](x,v)\ ds.
\end{eqnarray}
\end{lemma}
The proof of Lemma \ref{lem:distr-sol} is similar to that of Proposition 4.2 of \cite{Sant}
and is given in the Appendix.

\section{Existence of solutions to Cauchy problems for MDEs}\label{sec:Cauchy}
For simplicity we will focus on the set $\PP_c(\R^n)$ of probability measures
with compact support, but other sets with compactness properties
may be used, for instance based on bounds on the moments.
First we need to introduce some concepts from optimal transport theory.
We refer the reader to \cite{AGS,Sant,Villani,Villani2} for a complete perspective.

Given $(X,d)$ Polish space
and  given $\mu$, $\nu\in \PP(X)$
we indicate by $P(\mu,\nu)$ the set of transference plans from $\mu$ to $\nu$,
i.e. the set of probability measures on $X\times X$ with marginals
equal to $\mu$ and $\nu$ respectively.
Given $\tau\in P(\mu,\nu)$ let $J(\tau)$ bet its transportation cost:
\[
J(\tau)=\int_{X^2} d(x,y)\, d\tau(x,y).
\]
The Monge-Kantorovich or optimal transport problem amounts to find $\tau$
that minimizes $J(\tau)$ and the Wasserstein metric is defined by:
\[
W^X(\mu,\nu)=\inf_{\tau\in P(\mu,\nu)} J(\tau).
\]
For simplicity of notation we drop the superscript if $X=\R^n$.
We indicate by $P^{opt}(\mu,\nu)$ the (nonempty) set of optimal 
transference plans, i.e. minimizing $J(\tau)$, and
we always endow $\PP(X)$ with the Wasserstein metric and the relative topology.
We also recall the Kantorovich-Rubinstein duality:
\begin{equation}\label{eq:KR-duality}
W^X(\mu,\nu)=\sup\left\{\int_{X} f\,d(\mu-\nu)\ :f: X\to\R,\ Lip(f)\leq 1\right\},
\end{equation}
where $Lip(f)$ indicates the Lipschitz constant of $f$. We have the following:
\begin{lemma}\label{lem:unif-comp}
Consider a sequence $\mu_N\subset\PP_c(\R^n)$ and assume there exists $R>0$
such that $Supp(\mu_N)\subset B(0,R)$ .
Then there exists $\mu\in\PP_c(\R^n)$ and a subsequence, still indicated by $\mu_N$,
such that  $W(\mu_N(t),\mu(t))\to 0$.
\end{lemma}
The proof of the Lemma \ref{lem:unif-comp} is standard and we postpone it to the Appendix.\\
Our assumptions to prove existence of solutions are the following:
\begin{itemize}
\item[(H1)] $V$ is support sublinear for $\mu\in\PP_c(X)$, i.e. there exists $C>0$ such that
\[
\sup_{(x,v)\in Supp(V[\mu])} |v|\leq C
\left( 1 + \sup_{x\in Supp(\mu)} |x|\right).
\]
\item[(H2)] The map $V:\PP_c(\R^n)\to \PP_c(T\R^n)$ is continuous
(for the topology given by the Wasserstein metrics $W^{\R^n}$ and $W^{T\R^n}$.)
\end{itemize}
To prove existence of solutions to a Cauchy problem (\ref{eq:MDE-Cauchy}), 
we define a sequence of approximate solutions using a scheme of Euler type. 
We first introduce some more notation.\\
For $N\in\N$ let $\Delta_N =\frac{1}{N}$ be the time step size,
$\Delta^v_N=\frac{1}{N}$ the velocity step size and
$\Delta^x_N=\Delta^v_N\Delta_N=\frac{1}{N^2}$ the space step size.
We also define $x_i$ to be the $(2N^3+1)^n$ equispaced discretization points
of $\Z^n/(N^2)\cap [-N,N]^n$ 
and $v_j$ to be the $(2N^2+1)^n$ equispaced discretization points
of $\Z^n/N\cap [-N,N]^n$.
Given $\mu\in\PP_c(\R^n)$ we define the following operator providing
an approximation by finite sums of Dirac deltas:
\begin{equation}\label{eq:disc-x}
\A^x_N(\mu)=\sum_i m^x_i(\mu) \delta_{x_i}
\end{equation}
where 
\begin{equation}\label{eq:def:mx}
m^x_i(\mu)=\mu(x_i+Q)
\end{equation}
with $Q=([0,\frac{1}{N^2}[)^n$.
Similarly given $\mu\in\PP_c(\R^n)$ (whose support is contained in the set
$\Z^n/N\cap [-N,N]^n$), we set:
\begin{equation}\label{eq:disc-v}
\A^v_N(V[\mu])= \sum_i
\sum_j m^v_{ij}(V[\mu])\ \delta_{(x_i,v_j)}
\end{equation}
where 
\begin{equation}\label{eq:def:mv}
m^v_{ij}(V[\mu])=V[\mu](\{(x_i,v):v\in v_j+Q'\}),
\end{equation}
with $Q'=([0,\frac{1}{N}[)^n$. 
For every $\mu\in\PP_c(\R^n)$ there exists $N$ such that
$Supp(\mu_0)\subset [-N,N]^n$, thus
from the deifnition of $\A^x_N$ and $\A^v_N$, we easily get:
\begin{lemma}\label{lem:As}
Given $\mu\in\PP_c(\R^n)$, for $N$ sufficiently big the following holds:
\[
W(\A^x_N(\mu),\mu)\leq \Delta^x_N,\qquad
W^{T\R^n}(\A^v_N(V[\mu]),V[\mu])\leq \Delta^v_N.
\]
\end{lemma}
We are now ready to define a sequence of approximate solutions.
\begin{definition}\label{def:LAS} 
Consider $V$ satisfying (H1).
Given the Cauchy Problem (\ref{eq:MDE-Cauchy}), $T>0$ and $N\in\N$,
we define the Lattice Approximate Solution (LAS)
$\mu^N:[0,T]\to \PP_c(\R^n)$ as follows.\\
Recalling (\ref{eq:disc-x})-(\ref{eq:def:mv}), 
we set $\mu_0^N=\A^x_N(\mu_0)$ and, by recursion, define:
\begin{equation}\label{eq:def-rec}
\mu^N_{\ell+1}=\mu^N((\ell+1)\Delta_N)=\sum_i \sum_j 
m^v_{ij}(V[\mu^N(\ell\Delta N)])\ \delta_{x_i+\Delta_N\, v_j}.
\end{equation}
By definition of $\Delta_N$, $\Delta^v_N$, $\Delta^x_N$ and (\ref{eq:def-rec}),
$Supp(\mu^N_\ell)$ is  contained in the set $\Z^n/(N^2)\cap [-N,N]^n$, thus
we can write $\mu^N_\ell=\sum_i m^{N,\ell}_i \delta_{x_i}$
for some $m^{N,\ell}_i\geq 0$.
Finally $\mu^N$ is defined for all times by time-interpolation:
\begin{equation}\label{eq:def-mu-int}
\mu^N(\ell\Delta_N+t)=\sum_{ij}
m^v_{ij}(V[\mu^N(\ell\Delta N)])\ \delta_{x_i+t\, v_j}.
\end{equation}
\end{definition}
In other words, to define $\mu^N_{\ell+1}$ 
we approximate $V[\mu_\ell^N]$ by $\A^v_N(V[\mu_\ell^N])$
and use the corresponding velocities to move the Dirac deltas of $\mu_\ell^N$.
By definition of $V[\cdot]$, 
$\sum_j m^v_{ij}(V[\mu^N_\ell])=m^{N,\ell}_i$ thus the mass is conserved.
Notice that the support of $\mu^N(t)$ is not, in general,
contained in $\Z^n/(N^2)\cap [-N,N]^n$.
Because of assumption (H1), the support of $\mu^N_\ell$ keeps uniformly bounded
on the time interval $[0,T]$, as detailed in next Lemma.
\begin{lemma} \label{lem:bound-supp}
Given a PVF $V$ satisfying (H1),  $\mu_0$ with $Supp(\mu_0)\subset B(0,R)$ 
and $\ell$ such that $\ell\Delta_N\leq T$, the following holds true:
\begin{equation}
Supp(\mu^N_\ell) \subset B\left(0,e^{C\ell\Delta_N}R+e^{C\ell\Delta_N}\right)
\subset B\left(0,e^{C T} (R+1)\right). 
\end{equation}
\end{lemma}
We can now state the main result of this Section:
\begin{theorem}\label{th:MDE-ex}
Given a PVF $V$ satisfying (H1) and (H2),  
for every $T>0$ and $\mu_0\in\PP_c(\R^n)$
there exists a solution $\mu:[0,T]\to\PP_c(\R^n)$
to the Cauchy problem (\ref{eq:MDE-Cauchy}) 
obtained as uniform-in-time limit of LASs for the Wasserstein metric.\\
Moreover, if $Supp(\mu_0)\subset B(0,R)$ then:
\begin{equation}\label{eq:Lip-time}
W(\mu(t),\mu(s))\leq C\,e^{C T}\, (R+1)\ |t-s|.
\end{equation}
\end{theorem}
\begin{proof}
We have:
\[
W(\mu^N_{\ell+1},\mu^N_\ell)=W\left(\sum_{i,j} 
m^v_{ij}(V[\mu^N_\ell])\ \delta_{x_i+\Delta_N\, v_j},
\sum_i m^{N,\ell}_i \delta_{x_i}\right).
\]
Since $\sum_j m^v_{ij}(V[\mu^N_\ell])=m^{N,\ell}_i$, we can define a transference
plan from $\mu^N_\ell$ to $\mu^{N}_{\ell+1}$ by moving 
the mass $m^{N,\ell}_i \delta_{x_i}$
to $\sum_{j}  m^v_{ij}(V[\mu^N_\ell])\ \delta_{x_i+\Delta_N\, v_j}$.
Thus we obtain:
\[
W(\mu^N_{\ell+1},\mu^N_\ell)\leq 
\sum_{i,j} 
m^v_{ij}(V[\mu^N_\ell])\ |x_i+\Delta_N\, v_j-x_i|=
\Delta_N \sum_{i,j} 
m^v_{ij}(V[\mu^N_\ell])\ |v_j|.
\]
Let $R>0$ be such that $Supp(\mu_0)\subset B(0,R)$.
Using Lemma \ref{lem:bound-supp} and (H1), we deduce that 
$m^v_{ij}(V[\mu^N_\ell])\not= 0$ only if 
$|v_j|\leq C\,e^{C T}\, (R+1)$. Thus we get:
\[
W(\mu^N_{\ell+1},\mu^N_\ell)\leq C\,e^{C T}\, (R+1)\ \Delta_N.
\]
Repeating the same reasoning for $\mu(t)$ (see (\ref{eq:def-mu-int})) we get:
\begin{equation}\label{eq:Lip-time-approx}
W(\mu^N(t),\mu^N(s))\leq C\,e^{C T}\, (R+1)\ |t-s|.
\end{equation}
Therefore, the sequence $\mu^N:[0,T]\to \PP_c(\R^n)$ is uniformly Lipschiz
for the Wasserstein metric. By Ascoli-Arzel\'{a} Theorem, there exists a subsequence, 
still indicated by $\mu^N$, which converges uniformly to a Lipschitz curve 
$\mu:[0,T]\to \PP_c(\R^n)$ satisfying (\ref{eq:Lip-time}).

We now prove that the limit $\mu(t)$ satisfies (\ref{eq:MDE-int}).
Set $m^{N,\ell}_{ij}=m^v_{ij}(V[\mu^N_\ell])$, thus
$\sum_j m^{N,\ell}_{ij}=m^{N,\ell}_{i}$.
Given $f\in{\C}^\infty_c(\R^n)$ and $\bar{\ell}$, we compute:
\begin{eqnarray}
& \int_{\R^n}f\,d(\mu^N_{\bar\ell}-\mu^N_0)=
\int_{\R^n}f\,d\left(\sum_{\ell=0}^{\bar{\ell}-1}\,\big(\mu^N_{\ell+1}-\mu^N_\ell\big)\right)= \nonumber\\
& \sum_{\ell} \int_{\R^n}f\,d\left(\sum_{ij} m^{N,\ell}_{ij}\ \delta_{x_i+\Delta_N\, v_j} - \sum_i m^{N,\ell}_i\ \delta_{x_i}\right)=
\sum_{\ell}\sum_{ij} m^{N,\ell}_{ij} (f(x_i+\Delta_N\, v_j)-f(x_i))= \nonumber\\
& \sum_{\ell}\sum_{ij} m^{N,\ell}_{ij} \ \left[\Delta_N\ 
(\nabla f(x_i)\cdot v_j)+ \|f\|_{{\cal C}^2(B(0,C'))} o(\Delta_N)\right] \nonumber
\end{eqnarray}
where $C'=e^{C T} (R+1)$, so that $Supp(\mu^N_\ell)\subset B(0,C')$,
and $\lim_{N\to\infty}o(\Delta_N)/N=0$, thus
\[
= \sum_{\ell} \int_{\ell\Delta_N}^{(\ell+1)\Delta_N}\,\int_{T\R^n}
(\nabla f(x)\cdot v)\ d(\A^v_N(V[\mu^N_\ell]))(x,v)\ dt+
\|f\|_{{\cal C}^2(B(0,C'))}\ O(\Delta_N),
\]
where $\lim_{N\to\infty}O(\Delta_N)=0$.
Notice that $Supp(V[\mu^N_\ell])\subset B(0,C')\times B(0,C(1+C'))$.
If $x,y\in B(0,C')$ and $v,w\in B(0,C(1+C'))$, we can estimate:
\[
|\nabla f(x)\cdot v-\nabla f(y)\cdot w|\leq
|\nabla f(x)-\nabla f(y)|\ |v|+ |\nabla f(y)|\ |v-w|\leq
\]
\[
\|f\|_{{\cal C}^2(B(0,C'))}\ |x-y|\ C(1+C')+
\|f\|_{{\cal C}^1(B(0,C'))}\ |v-w| \leq
L(f) |(x,v)-(y,w)|,
\]
where $L(f)=\sqrt{2}\max\{\|f\|_{{\cal C}^2(B(0,C'))}\ C(1+C'),\|f\|_{{\cal C}^1(B(0,C'))}\}$. Thus $L(f)$ 
is a Lipschitz constant for the function $\nabla f(x)\cdot v$ on
$B(0,C')\times B(0,C(1+C'))$.
Define $\xi=\chi_{B(0,C')\times B(0,C(1+C'))}$ (the indicator function
of the set $B(0,C')\times B(0,C(1+C'))$, 
then using  the Kantorovich-Rubinstein duality (\ref{eq:KR-duality}) for $X=T\R^n$
and Lemma \ref{lem:As}, we get:
\begin{eqnarray}\label{eq:AvV-V}
& L(f)\left|\int_{T\R^n} \frac{\nabla f(x)\cdot v}{L(f)}
\ d(\A^v_N(V[\mu^N_\ell])-V[\mu^N_\ell])(x,v) \right|=\nonumber\\
& L(f)\left|\int_{T\R^n} \frac{(\nabla f(x)\cdot v)\xi(x,v)}{L(f)}
\ d(\A^v_N(V[\mu^N_\ell])-V[\mu^N_\ell])(x,v) \right|
\leq \nonumber\\
& L(f)\ W^{T\R^n}(\A^v_N(V[\mu^N_\ell]),V[\mu^N_\ell])
\leq L(f)\ \Delta^v_N.
\end{eqnarray}
In the same way, for every $t\in [\ell\Delta_N, (\ell+1)\Delta_N[$, it holds:
\[
L(f)\left|\int_{T\R^n} \frac{\nabla f(x)\cdot v}{L(f)}
\ d(V[\mu^N_\ell]-V[\mu^N(t)])(x,v) \right|
\leq L(f)\ W^{T\R^n}(V[\mu^N_\ell],V[\mu^N(t)]).
\]
Now, $W^{T\R^n}(V[\mu^N_\ell],V[\mu^N(t)])\leq 2C'+2C(1+C')$,
for every $t\in [\ell\Delta_N, (\ell+1)\Delta_N[$, and,
from (\ref{eq:Lip-time-approx}) and (H2), we have that 
$W^{T\R^n}(V[\mu^N_\ell],V[\mu^N(t)]) \to 0$ as $N\to\infty$.
Thus by Lebesgue dominated convergence:
\begin{equation}\label{eq:V-Vt}
\left|\sum_\ell\ \int_{\ell\Delta_N}^{(\ell+1)\Delta_N}\,
\left(\int_{T\R^n} (\nabla f(x)\cdot v)
\ d(V[\mu^N_\ell]-V[\mu^N(t)])(x,v)\right)\ dt \right|\to 0.
\end{equation}
From (\ref{eq:AvV-V}) and (\ref{eq:V-Vt}) we get:
\[
\int_{\R^n}f\,d(\mu^N_{\bar\ell}-\mu^N_0)=
\int_{0}^{\bar{\ell}\Delta_N} \,
\int_{T\R^n}(\nabla f(x)\cdot v)\ d(V[\mu^N(t)])(x,v)\ dt
+O(\Delta_N).
\]
Since $\bar{\ell}$ is arbitrary and the integrands are bounded, 
using (\ref{eq:Lip-time-approx}) we conclude:
\begin{equation}\label{eq:approx-sol-N}
\int_{\R^n}f\,d(\mu^N(t)-\mu^N_0)=
\int_{0}^{t} \,\int_{T\R^n}(\nabla f(x)\cdot v)\ d(V[\mu^N(t)])(x,v)\ dt+O(\Delta_N).
\end{equation}
Now, from  the Kantorovich-Rubinstein duality (\ref{eq:KR-duality}) and Lemma \ref{lem:As}, reasoning as above we get the following estimates:
\begin{eqnarray}
& \left|\int_{\R^n}f\,d(\mu^N_0(t)-\mu_0(t))\right|\leq\|f\|_{{\cal C}^1(B(0,C'))}  \ \Delta^x_N, \nonumber\\
& \left|\int_{\R^n}f\,d(\mu^N(t)-\mu(t))\right|\leq\|f\|_{{\cal C}^1(B(0,C'))}  \ W(\mu^N(t),\mu(t)), \nonumber\\
& \left| \int_0^{t} (\nabla f(x)\cdot v)\ 
d(V[\mu^N(s)]-V[\mu(s)])\ ds \right|\leq L(f)
\ \int_0^{t} W^{T\R^n}(V[\mu^N(s)],V[\mu(s)])\ ds.\nonumber
\end{eqnarray}
The integral argument on the right-hand side (of last inequality)
is bounded and, by (H2), converges  to zero as $N$ tends to infinity .
Thus, by Lebesgue dominated convergence, it tends to zero.
Passing to the limit in (\ref{eq:approx-sol-N}),
we have that $\mu(t)$ is a solution to (\ref{eq:MDE-Cauchy}).
\end{proof}

\section{Lipschitz semigroup of solutions to MDEs}\label{sec:Lip}
We now investigate continuous dependence from initial data.
More precisely, we provide a new condition ensuring the existence of
a Lipschitz semigroup of solutions obtained as limit of LASs.\\
Notice that $V[\mu]$ is supported on $T\R^n$ but the two components $(x,v)$ have
different meanings, indeed $v$ represents a tangent vector thus an infinitesimal displacement.
For this reason, instead of using $W^{T\R^n}$, we are going to introduce another
concept to measure distances among elements of $\PP(T\R^N)$.
\begin{definition}
Consider $V_i\in\PP_c(\R^n)$, $i=1,2$, and denote by $\mu_i$
the marginal over the base, i.e. $\pi_1\# V_i=\mu_i$.
We define the following quantity:
\begin{equation}\label{eq:WW}
\WW (V_1,V_2)=\inf \left\{
\int_{T\R^n\times T\R^n} |v-w| \ dT(x,v,y,w)\ :
T\in P(V_1,V_2), \ \pi_{13}\# T\in P^{opt}(\mu_1,\mu_2) \right\}
\end{equation}
\end{definition}
The condition $\pi_{13}\# T\in P^{opt}(\mu_1,\mu_2)$ tells us that $T$ acts optimally on the base transporting $\mu_1$ to $\mu_2$. Therefore $\WW$
gives the optimal transport distance of the fiber components 
based on optimal ways to transport the marginals on the base.
\begin{remark}
Notice that $\WW$ is not a metric  since it can vanish 
for distinct elements of $\PP(T\R^N)$. It would be tempting to add the term $|x-y|$ 
to the integrand in (\ref{eq:WW}) (or a norm of $(x,v,y,w)$) but
we would not obtain a metric, because the triangular inequality does not hold.
A simple example is obtained by setting in $T\R^2$:
$V_1=\frac12 \delta_{((0,0),(1,0))}+\frac12  \delta_{((1,0),(3,0))}$,
$V_2=\frac12 \delta_{((0,1),(1,0))}+\frac12  \delta_{((1,-1),(3,0))}$ and
$V_3=\frac12 \delta_{((1,1),(1,0))}+\frac12  \delta_{((0,-1),(3,0))}$.
There exists a unique optimal transference plan
from $\mu_1=\pi_1\#V_1$ to $\mu_2 = \pi_1\#V_2$,
moving the mass from $(0,0)$ to $(0,1)$ and from $(1,0)$ to $(1,-1)$,
a unique optimal transference plan from $\mu_2$ to $\mu_3=\pi_1\#V_3$,
moving the mass from $(0,1)$ to $(1,1)$ and from $(1,-1)$ to $(0,-1)$,
and a unique optimal transference plan from $\mu_1$ to $\mu_3$,
moving the mass from $(0,0)$ to $(0,-1)$ and from $(1,0)$ to $(1,1)$.
Then, the set  ${\cal T}_{12}=
\{T \in P(V_1,V_2) : \pi_{13}\#T \in P^{opt}(\mu_1,\mu_2) \}$
has a unique element and 
$\inf_{T\in {\cal T}_{12}} \int (|x-y|+|v-w|) dT = 1$.
Defining similarly ${\cal T}_{23}$ and ${\cal T}_{13}$, we get
$\inf_{T\in {\cal T}_{23}} \int (|x-y|+|v-w|) dT = 1$ and
$\inf_{T\in {\cal T}_{13}} \int (|x-y|+|v-w|) dT = 3$.
\end{remark}
We are now ready to state a new assumption, which is a local Lipschitz-type
condition on the map $\mu\to V[\mu]$ for $\WW$.
We require that:
\begin{itemize}
\item[(H3)] For every $R>0$ there exists $K=K(R)>0$ such that if 
$Supp(\mu),Supp(\nu)\subset B(0,R)$ then
\begin{equation}
\WW(V[\mu],V[\nu])\leq K\ W(\mu,\nu).
\end{equation}
\end{itemize}
The quantity $\WW$ in general can not compare to $W^{T\R^n}$, which weights
in the same way the base and the fiber. However, we have the following:
\begin{lemma}\label{lem:WT}
Given $\mu$, $\nu\in\PP(\R^n)$, it holds 
\begin{equation}
W^{T\R^n}(V[\mu],V[\nu])\leq  \WW(V[\mu],V[\nu]) + W(\mu,\nu).
\end{equation}
In particular (H3) implies local Lipschitz continuity of $V$ w.r.t. $W^{T\R^n}$.
\end{lemma}
\begin{proof}
By definition we have:
\begin{eqnarray}
& W^{T\R^n}(V[\mu],V[\nu])=\inf_{T\in P(V[\mu],V[\nu])} \int_{(T\R^n)^2}
|(x,v)-(y,w)| \ dT(x,v,y,w)\leq \nonumber\\
& \inf_{T\in P(V[\mu],V[\nu])} \int_{(T\R^n)^2} (|x-y|+|v-w|) \ dT(x,v,y,w)\leq \nonumber\\
& \inf_{T\in P(V[\mu],V[\nu]), \pi_{13}\# T\in P^{opt}(\mu,\nu)} \int_{(T\R^n)^2} (|x-y|+|v-w|) \ dT(x,v,y,w)= \nonumber\\
& W(\mu,\nu)+\WW(V[\mu],V[\nu]).\nonumber
\end{eqnarray}
\end{proof}
\begin{remark}
The converse of Lemma \ref{lem:WT} does not hold true,
since $\WW$ can not be estimated in terms of $W$ and $W^{T\R^n}$.
To see this consider the PVF on $\R$ ($n=1$) defined by
$V[\mu]=\mu\otimes \varphi(\mu,x) \lambda$ with
$\lambda$ the Lebesgue measure on $\R$ and
\[
\varphi(\mu,x)=\sin\left(\frac{2\pi\ x}{Var(\mu)}\right)
\]
where $Var(\mu)=\int (x-\bar{x})^2\ d\mu$, $\bar{x}=\int x\,d\mu$, 
is the variance of $\mu$. Let $\mu_m=m\chi_{[0,\frac{1}{m}]}\,\lambda$.
Then we easly compute $\int x\,d\mu_m=\frac{1}{2m}$ and
$Var(\mu_m)=\frac{1}{12m^2}$, thus $\varphi(\mu_m,x)$  
has period $\frac{1}{12m^2}$.
Define also $\mu_m'=m\chi_{[\frac{1}{24m^2},\frac{1}{m}+\frac{1}{24m^2}]}
\,\lambda$. The Wasserstein distance between 
$\mu_{m}$ and $\mu_{m}'$ is realized by the map $T(x)=x+\frac{1}{24m^2}$,
see Theorem 2.18 and Remark 2.19 in \cite{Villani}. Therefore we get:
\[
W(\mu_{m},\mu_m')=\frac{1}{24m^2}.
\]
Since $\varphi(\mu_m,x)$  has period $\frac{1}{12m^2}$,
the map $T$ sends the intervals where $\varphi(\mu_{m},x)$ is positive 
to intervals where $\varphi(\mu_m',x)$ is negative and viceversa.
Therefore we easily compute:
\[
\WW(V[\mu_{m}],V[\mu_{m}'])= m\,\int_0^{\frac{1}{m}}
2\ |\sin(24\,m^2\,\pi\,x)|\, dx=\frac{4}{\pi},
\]
and assumption (H3) does not hold true.\\
On the other side, to estimate $W^{T\R^n}(V[\mu_{m}],V[\mu_m'])$,
consider the map $T(x)=x+\frac{1}{12m^2}$ sending
the intervals where $\varphi(\mu_{m},x)$ is positive 
to intervals where $\varphi(\mu_m',x)$ is also positive.
To be precise we have also to send the last inteval 
$[\frac{1}{m}-\frac{1}{24m^2},\frac{1}{m}]$ of the domain of $\mu_m$ to the first
interval $[\frac{1}{24m^2},\frac{1}{12m^2}]$ of the domain of $\mu_m'$.
Therefore we get:
\[
W^{T\R^n}(V[\mu_{m}],V[\mu_m'])\leq 
\frac{1}{12m^2}+\frac{1}{24m^2}.
\]
\end{remark}

We are now ready to prove existence of semigroups of solutions.
First we give the following:
\begin{definition}\label{def:Lip-semigroup}
Consider a PVF $V$ satisfying (H1) and $T>0$.
A Lipschitz semigroup for (\ref{eq:MDE})
is a map $S:[0,T]\times \PP_c(\R^n) \to\PP_c(\R^n)$ such that
for every $\mu,\nu \in \PP_c(\R^n)$ and $t,s\in [0,T]$ the following holds:
\begin{itemize}
\item[i)] $S_0\mu=\mu$ and $S_t\,S_s\,\mu=S_{t+s}\,\mu$;
\item[ii)] the map $t\mapsto S_t\mu$ is a solution to (\ref{eq:MDE});  
\item[iii)] for every $R>0$ there exists $C(R)>0$ such that
if $Supp(\mu), Supp(\nu) \subset B(0,R)$ then:
\begin{equation}\label{eq:supp-bound}
Supp(S_t\mu)\subset B(0,e^{Ct} (R+1))
\end{equation}
\begin{equation}\label{eq:cont-dep}
W(S_t\mu,S_t\nu)\leq e^{C(R)t}W(\mu,\nu),
\end{equation}
\begin{equation}\label{eq:Lip-time-cont}
W(S_t\mu,S_s\mu)\leq C(R)\ |t-s|.
\end{equation}
\end{itemize}
\end{definition}
Next Theorem provides existence of a Lipschitz semigroup of solutions to an MDE, 
obtained via limit of LAS:
\begin{theorem}\label{th:MDE}
Given $V$ satisfying (H1) and (H3), and $T>0$, 
there exists a Lipschitz semigroup of solutions to (\ref{eq:MDE}),
obtained passing to the limit in LASs.
\end{theorem}
\begin{proof}
We first prove Wasserstein estimates on LASs for different initial data.
Fix $\mu_0,\nu_0\in\PP_c(\R^n)$ and call $\mu^N$, respectively $\nu^N$, the LAS defined using $\mu_0$, respectively $\nu_0$, as initial datum.
First, from Lemma \ref{lem:As} we get:
\begin{equation}\label{eq:in-est}
W(\mu_0^N,\nu_0^{N})=W(\A^x_N(\mu_0),\A^x_{N}(\nu_0))\leq 
W(\mu_0,\nu_0)+2\,\Delta^x_{N}.
\end{equation}
Let us now estimate the Wasserstein distance between $\mu^N_\ell$
and $\nu^{N}_{\ell}$ by recursion:
\[
W(\mu^N_{\ell+1},\nu^{N}_{\ell+1})=
W\left( \sum_{i,j} 
m^v_{ij}(V[\mu^N_\ell])\ \delta_{x_i+\Delta_{N}\, v_j},
\sum_{i.j} 
m^v_{ij}(V[\nu^{N}_{\ell}])\ \delta_{x_i+\Delta_{N}\, v_j}\right).
\]
Let $R>0$ be such that $Supp(\mu_0), Supp(\nu_0)\subset B(0,R)$.
By Lemma \ref{lem:bound-supp}, the supports of $\mu^N$ and $\nu^N$
are uniformly contained in $B(0,e^{Ct}(R+1))$, thus, by assumption (H3), 
there exists $K=K(e^{Ct}(R+1))>0$ such that:
\[
\WW(V[\mu^N_\ell],V[\nu^{N}_{\ell}])
\leq K\,W(\mu^N_\ell, \nu^{N}_{\ell})
\]
thus there exists 
$T\in P(V[\mu^N_\ell],V[\nu^{N}_{\ell}])$ such that: 
\begin{equation}\label{eq:T-est}
\int_{T\R^n\times T\R^n} |v-w| \ dT(x,v,y,w)\leq 
K\ W(\mu^N_\ell, \nu^{N}_{\ell})+\Delta_N,
\end{equation}
and $\pi_{13}\# T\in P^{opt}(\mu^N_\ell, \nu^{N}_{\ell})$.
We will now construct a transference plan from
$\mu^N_{\ell+1}$ to $\nu^{N}_{\ell+1}$ by moving masses
using the plan $T$. More precisely, 
define $\tau_{ij}\in \PP_c((\R^n)^2)$ by
$\tau_{ij}(A,B)=T(\{(x_i,v,x_j,w):x_i+\Delta_{N}\,v\in A, 
x_j+\Delta_{N}\,w\in B\})$. 
In other words if $T$  moves a mass from $\delta_{(x_i,v)}$
to $\delta_{(x_j,w)}$ then $\tau_{ij}$ moves the same mass from $\delta_{x_i+\Delta_{N}\,v}$ to $\delta_{x_j+\Delta_{N}\,w}$.
Defining $\tau=\sum_{ij} \tau_{ij}\in P(\mu^N_{\ell+1},\nu^{N}_{\ell+1})$, we get:
\begin{eqnarray}
& W(\mu^N_{\ell+1},\nu^{N}_{\ell+1})\leq \int_{\R^n\times \R^n} |x-y|\ d\tau(x,y)=\nonumber\\
& \int_{T\R^n\times T\R^n} |(x+\Delta_{N}\,v)-(y+\Delta_{N}\,w)|\ dT(x,v,y,w)\leq \nonumber\\
& \int_{T\R^n\times T\R^n} |x-y|\ dT(x,v,y,w)+ \int_{T\R^n\times T\R^n} \Delta_{N}\,|v-y|\ dT(x,v,y,w)
\dot= I_1+I_2.\nonumber
\end{eqnarray}
Since $\pi_{13}\# T\in P^{opt}(\mu^N_\ell, \nu^{N}_{\ell})$,
we have:
\[
I_1 = W(\mu^N_\ell, \nu^{N}_{\ell}),
\]
while by (\ref{eq:T-est}) we get:
\[
I_2\leq \Delta_{N}\ \left(
K\ W(\mu^N_\ell, \nu^{N}_{\ell})+\Delta_N\right).
\]
Finally it holds:
\begin{equation}\label{eq:rec-est12}
W(\mu^N_{\ell+1},\nu^{N}_{\ell+1})\leq
(1+K\,\Delta_{N})\ W(\mu^N_\ell, \nu^{N}_{\ell})+\Delta_{N}^2.
\end{equation}
Combining (\ref{eq:in-est}) and (\ref{eq:rec-est12}), we get:
\begin{eqnarray}\label{eq:fin-est}
W(\mu^N_\ell,\nu^{N}_{\ell})
&\leq & (1+K\,\Delta_{N})^{\ell} \ (W(\mu_0,\nu_0)+\Delta^x_{N}) +
\sum_{k=0}^{\ell-1} (1+K\,\Delta_{N})^k\ \Delta_N^2 \nonumber\\
&\leq & e^{K\,\ell\Delta_{N}}\ (W(\mu_0,\nu_0)+\Delta^x_{N})  
+ \frac{e^{K\,(\ell-1)l\Delta_{N}}-1}{K}\ \Delta_{N}\nonumber\\
&\leq & e^{K\,\ell\Delta_{N}}\ \left(W(\mu_0,\nu_0)+\Delta^x_{N}+\frac{\Delta_N}{K}\right).
\end{eqnarray}
Now, define the countable set 
$\DD^q=\{\mu_0\in \PP_c(\R^n): \mu_0=\sum_{i=1}^N m_i\delta_{x_i},
N\in\N, 0<m_i\in\Q, \sum_i m_i=1, x_i\in\Q^n \}$. 
By Lemma \ref{lem:WT}, hypothesis (H3) implies (H2), thus for every $\mu_0\in\DD^q$, 
we can apply Theorem \ref{th:MDE-ex} and find a subsequence of $\mu^N$ which converges 
uniformly on $[0,T]$ for the Wasserstein metric to a solution satisfying (\ref{eq:Lip-time}). 
Using a diagonal argument we find a subsequence, still indicated by $\mu^N$, which converges uniformly
on $[0,T]$ for every $\mu_0\in\DD^q$ to a solution $S_t\mu_0$. Moreover,
given $\mu,\nu\in\DD^q$, with $Supp(\mu),Supp(\nu)\subset B(0,R)$,
passing to the limit in (\ref{eq:fin-est}) and using (\ref{eq:Lip-time}), we have for $K=K(e^{CT}(R+1))$:
\begin{equation}\label{eq:cont-dep-on-D}
W(S_t\mu,S_t\nu)\leq e^{Kt} W(\mu,\nu).
\end{equation}
By (\ref{eq:cont-dep-on-D}) and the density of $\DD^q$ in $\PP_c(\R^n)$, we can uniquely extend
the map $S$ to the whole set $\PP_c(\R^n)$ by approximation.
Using the Kantorovich-Rubinstein duality, as in the proof of Theorem \ref{th:MDE-ex}, 
we conclude that ii) of Definition \ref{def:Lip-semigroup} holds true for $S$  
on the whole set $\PP_c(\R^n)$. 
Moreover, again by approximation, we get that (\ref{eq:Lip-time}) and (\ref{eq:cont-dep-on-D}) 
hold on the whole set  $\PP_c(\R^n)$, thus $S$ satisfies also iii).\\
Let us now prove i) of Definition \ref{def:Lip-semigroup}. From (\ref{eq:in-est}), we get $S_0\mu=\mu$.
Consider $\mu\in\PP_c(\R^n)$ and $t,s\in [0,T]$.
We use the notation $\mu^{N}_\ell(\mu)$ to indicate the LAS defined
having $\mu$ as initial datum. Then, for every $\epsilon$ there exists $N$
such that if $\ell=\lfloor Ns\rfloor$ (where $\lfloor s\rfloor = \sup \{n\in\N: n\leq s\}$ is 
the usual floor function) and $\ell'\lfloor Nt\rfloor$ then:
\begin{equation}\label{eq:i-first}
W(S_s\mu,\mu^N_\ell(\mu))\leq\epsilon,\quad
W(S_t\mu^N_\ell(\mu),\mu^{N}_{\ell'}(\mu^N_\ell(\mu)))\leq\epsilon.
\end{equation}
Notice that $|\ell+\ell'- \lfloor N(t+s)\rfloor|\leq 1$, thus, possibly changing $N$,
we also get:
\begin{equation}\label{eq:i-sec}
W(\mu^{N}_{\ell+\ell'}(\mu),S_{t+s}\mu)\leq\epsilon.
\end{equation}
By definition, $\mu^{N}_{\ell'}(\mu^N_\ell(\mu))= \mu^{N}_{\ell+\ell'}(\mu)$, thus using 
(\ref{eq:cont-dep-on-D}), (\ref{eq:i-first}) and (\ref{eq:i-sec}) we estimate:
\begin{eqnarray}
& W(S_tS_s\mu,S_{t+s}\mu)\leq W(S_tS_s\mu,S_t\mu^N_\ell(\mu))+
 W(S_t\mu^N_\ell(\mu),\mu^{N}_{\ell'}(\mu^N_\ell(\mu)))+\nonumber\\
& W(\mu^{N}_{\ell'}(\mu^N_\ell(\mu)),\mu^{N}_{\ell+\ell'}(\mu_))+
W(\mu^{N}_{\ell+\ell'}(\mu),S_{t+s}\mu)\leq e^{Kt} W(S_s\mu,\mu^N_\ell(\mu))+2\epsilon
\leq e^{Kt}\epsilon +2\epsilon.\nonumber
\end{eqnarray}
For the arbitrariety of $\epsilon$, also i) holds true for $S$ and the proof is complete.
\end{proof}

\section{Uniqueness of solutions semigroup to MDEs}\label{sec:Dirac}
Definition \ref{def:sol-MDE} is not expected to guarantee uniqueness in general,
see Example \ref{ex:non-uni} below.
However we can obtain uniqueness of a Lipschitz semigroup
prescribing the small-time evolution of finite sums of Dirac deltas.

We first define the concept of Dirac germ.
\begin{definition}\label{def:Dir-germ}
Consider a PVF $V$ satisfying (H1) and define
$\DD=\{\mu\in \PP_c(\R^n): \mu=\sum_{i=1}^N m_i\delta_{x_i},
N\in\N, 0<m_i, \sum_i m_i=1, x_i\in\R^n \}$.
A Dirac germ $\gamma$ is a map assigning to every $\mu\in\DD$
a Lipschitz curve $\gamma_\mu:[0,\epsilon(\mu)]\to \PP_c(\R^n)$, 
$\epsilon(\mu)>0$ uniformly positive for uniformly bounded supports,
such that $\gamma_\mu(0)=\mu$ and $\gamma_\mu$ is a solution to
(\ref{eq:MDE}).
\end{definition}
In rough words a Dirac germ is a prescribed evolution of solutions
for finite sums of Dirac deltas (for sufficiently small times).
\begin{definition}\label{def:Dir-semigroup}
Consider a PVF $V$ satisfying (H1), $T>0$ and a Dirac germ $\gamma$. 
A Dirac semigroup, compatible with the Dirac germ $\gamma$,
is a Lipschitz semigroup $S:[0,T]\times \PP_c(\R^n) \to\PP_c(\R^n)$ 
for (\ref{eq:MDE}) such that the following holds.
For every $R>0$, denoting
$\DD_R=\{\mu\in \DD: Supp(\mu)\subset B(0,R)\}$,
there exists $C(R)>0$ such that
for every $t\in [0,\inf_{\mu\in\DD_R}\epsilon(\mu)]$ it holds:
\begin{equation}
\sup_{\mu\in\DD_R} W(S_t\mu,\gamma_{\mu}(t))\leq C(R)\ t^2.
\end{equation}
\end{definition}
In other words $S$ is a Lipschitz semigroup whose trajectories are well approximated by the Dirac germ.
To prove uniqueness of a Dirac semigroup (compatible with a given Dirac germ),
we use the following Lemma:
\begin{lemma}\label{lem:Sem-est}
Let $S$ be a Lipschitz semigroup and $\mu:[0,T]\to\PP_c(\R^n)$ a Lipschitz continuous curve,
then we have:
\[
W(S_t \mu(0),\mu(t))\leq e^{Kt} \int_0^t
\liminf_{h\to 0+} \frac{1}{h} W(S_h \mu(s),\mu(s+h))\ ds.
\]
\end{lemma}
Lemma \ref{lem:Sem-est} was proved in \cite{Bressan-book} (Theorem 2.9) for semigroups on
Banach spaces, but is valid also for metric spaces. For reader convenience we
detail the proof in the Appendix.

We are now ready to prove the following:
\begin{theorem}\label{th:uniq-sem}
Consider a PVF $V$ satisfying (H1), $T>0$ and a Dirac germ $\gamma$.
There exists at most one Lipschitz semigroup compatible with $\gamma$.
\end{theorem}
\begin{proof}
Let $S^1$, $S^2$ be two Lipschitz semigroups compatible with $\gamma$.
By Lemma \ref{lem:Sem-est}, we have for every $\mu\in\PP_c(\R^n)$:
\[
W(S^1_t \mu,S^2_t\mu)\leq e^{Kt} \int_0^t
\liminf_{h\to 0+} \frac{1}{h} W(S^1_h S^2_s\mu,S^2_{s+h}\mu)\ ds.
\]
There exists $R>0$ such that $Supp(\mu)\subset B(0,R)$, thus 
from (\ref{eq:supp-bound}) of Definition \ref{def:Lip-semigroup},
$Supp(S^i_t\mu)\subset B(0,e^{CT}(R+1))$ for every $t\in [0,T]$.
Moreover, there exists $C_i=C_i(e^{CT}(R+1))>0$ 
such that (\ref{eq:cont-dep}) and (\ref{eq:Lip-time-cont}) holds true.\\
Now, fix $s$.  For every $\epsilon$ there exists $\mu_s\in\DD$ such that 
$W(\mu_s,S^2_s\mu)\leq\epsilon$ and, from Definition \ref{def:Dir-semigroup},
there exists $C=C(e^{CT}(R+1))$ such that:
\[
W(S^1_h \mu_s,\gamma_{\mu_s}(h))\leq C \ h^2,\qquad
W(S^2_h \mu_s,\gamma_{\mu_s}(h))\leq C \ h^2.
\]
Then:
\[
W(S^1_h S^2_s\mu,S^2_{s+h}\mu)=W(S^1_h S^2_s\mu,S^2_{h}S^2_s\mu)\leq
\]
\[
W(S^1_h S^2_s\mu,S^1_h \mu_s)+W(S^1_h \mu_s,\gamma_{\mu_s}(h))+
W(\gamma_{\mu_s}(h),S^2_h \mu_s)+W(S^2_h \mu_s,S^2_{h}S^2_s\mu)\leq
\]
\[
2\ (e^{\tilde{C}h}\epsilon+C\,h^2)
\]
where $\tilde{C}=\max\{C_1,C_2\}$.
Since $\epsilon$ is arbitrary, we conclude:
\[
\liminf_{h\to 0+} \frac{1}{h} W(S^1_h S^2_s\mu,S^2_{s+h}\mu)= 0
\]
which gives $W(S^1_t \mu,S^2_t\mu)=0$.
\end{proof}

Our main result about uniqueness is the following:
\begin{theorem}\label{th:uniq-sem-LAS}
Let $V$ be a PVF satisfying (H1) and (H3). If for every $\mu\in\DD$
(finite sum of Dirac deltas)
the sequence of LASs $\mu^N$ converge to a unique limit, then
there exists a unique Lipschitz semigroup whose trajectories are limits of LASs.
\end{theorem}
\begin{proof}
By Theorem \ref{th:MDE} there exists a Lipschitz semigroup $S$
obtained via limits of LASs. Moreover, we can define a Dirac germ $\gamma$
from the unique limits of LASs for $\mu\in\DD$ and,
by definition, $S$ is compatible with $\gamma$.\\
If $S'$ is another Lipschitz semigroup whose trajectories are limits of LASs,
then, by uniqueness of limits, $S'$ is compatible with $\gamma$, thus by
Theorem \ref{th:uniq-sem} coincides with $S$.
\end{proof}
Theorem \ref{th:uniq-sem-LAS} allows to reduce the question of uniqueness
of a Lipschitz semigroup (compatible with LAS limits) to that of understanding
uniqueness of LAS limits for finite sums of Dirac deltas.
The latter question is much simpler than the general uniqueness
of Lipschitz semigroups and, in next sections, we provide examples of PVFs for which we can apply Theorem \ref{th:uniq-sem-LAS}.

\section{Ordinary differential equations and MDEs}\label{sec:ODEs}
In this section we show natural connections between Ordinary Differential Equations
(briefly ODEs) and MDEs. 
We start with the following definition:
\begin{definition}
Consider an ODE $\dot x = v(x)$, $v:\R^n\to\R^n$. 
We define a PVF $V^v$ by:
\[
V^v[\mu]=\mu\otimes \delta_{v(x)}.
\]
\end{definition}
The main question is if $V^v$ satisfies hypothesis (H1) and (H2) or (H3).\\ 
Notice that (H1) easily follows
from a sublinear growth requirement on $v$, i.e. there exists $C>0$ such
that $|v(x)|\leq C(1+|x|)$. However, the continuity of the map
$\mu\to V^v[\mu]$ is not implied by the continuity of $v$.
This can be easily seen as follows. Consider two measures
$\mu^i=\sum_{j=1}^{N_i}\, m^i_j\,\delta_{x^i_j}$, $i=1,2$. 
By possibly splitting the masses $m^i_j$, we can assume $N_1=N_2=N$ and that there exists
a map $\sigma:\{1,\ldots,N\}\to \{1,\ldots,N\}$
so that an optimal transference plan between $\mu^1$ and $\mu^2$ is equivalent to move
the mass $m^1_j=m^2_{\sigma(j)}$ from $x^1_j$ to $x^2_{\sigma(j)}$.
Then we have:
\[
W^{T\R^n}(\mu^1,\mu^2)\sim \sum_{j=1}^{N} m^1_j\ (|x^1_j-x^2_{\sigma(j)}|^2+
|v(x^1_j)-v(x^2_{\sigma(j)})|^2)^{\frac12},
\]
therefore continuity follows only if we can estimate:
$\sum_{j=1}^{N} m^1_j |v(x^1_j)-v(x^2_{\sigma(j)})|$
in terms of $\sum_{j=1}^{N} m^1_j |x^1_j-x^2_{\sigma(j)}|$
which requires higher regularity.\\
We first prove the following:
\begin{prop}\label{prop:ODE-H3}
$V^v$ satisfies (H3) for finite sums of Dirac deltas if and only if $v$ 
is locally Lipschitz continuous. 
\end{prop}
\begin{proof}
Assume first $v$ to be locally Lipschitz, fix $R>0$ and let $L(v,R)$
be the Lipschitz constant of $v$ on $B(0,R)$.
Consider two probability measures
$\mu^i=\sum_{j=1}^{N_i}\, m^i_j\,\delta_{x^i_j}$, $i=1,2$. As above,
we can assume $N_1=N_2=N$ and that there exists
a map $\sigma:\{1,\ldots,N\}\to \{1,\ldots,N\}$
so that an optimal transference plan between $\mu^1$ and $\mu^2$ moves
the mass $m^1_j=m^2_{\sigma(j)}$ from $x^1_j$ to $x^2_{\sigma(j)}$.
Then it holds:
\[
\WW (V[\mu^1],V[\mu^2])=
\WW \left(V^v\left[\sum_{j=1}^N\, m^1_j\,\delta_{x^1_j}\right],V^v\left[\sum_{j=1}^N\, m^2_j\,\delta_{x^2_j}\right]\right)
\leq 
\]
\[
\leq \sum_{j=1}^N m^1_j \ |v(x^1_j)-v(x^2_{\sigma(j)})|
\leq L(v,R)\,\sum_{j=1}^N m^1_j |x^1_j-x^2_{\sigma(j)}|
=L(v,R)\ W(\mu^1,\mu^2).
\]
Conversely, assume $V^v$ to satisfy (H3). Take two points $x,y\in B(0,R)$,
then we have:
\[
|v(x)-v(y)|= \WW(\delta_{(x,v(x))},\delta_{(y,v(y))})
\leq K\, W(\delta_x,\delta_y)=K\,|x-y|
\]
thus we conclude that $v$ is locally Lipchitz continuous.
\end{proof}
From uniqueness of solutions for locally Lipschitz vector fields we obtain the following:
\begin{theorem}\label{th:ODE}
If $v$ is locally Lipschitz with sublinear growth
(i.e. there exists $C>0$ such that $|v(x)|\leq C(1+|x|)$), then $V^v$ satisfies (H3),
thus the MDE, associated to $V^v$, admits a unique Lipschitz semigroup
obtained as limit of LASs.
\end{theorem}
\begin{proof}
Property (H1) for $V^v$ follows from the sublinear growth of $v$,
while (H3) follows from Proposition \ref{prop:ODE-H3} for finite sums of Dirac deltas.
Let us prove (H3) for $V^v$ on the whole set $\PP_c(\R^n)$.
Consider $\mu_i\in\PP_c(\R^n)$, $i=1,2$,
let $x_j=(x^j_1,\ldots,x^j_n)$ be the points of $\Z^n/N$ and define:
\[
\mu_i^N=\sum_j m^i_j \delta_{b_j},
\]
where $m^i_j=\mu_i(x_j+Q_N)$, $Q_N=\{y:0\leq y_i< \frac{1}{N},i=1,\ldots,n\}
\subset\R^n$, and $b_j=(x^j_1+\frac{1}{2N},\ldots,x^j_n+\frac{1}{2N})$
is the baricenter of $x_j+Q_N$. 
An optimal transference plan between $\mu_i$ and $\mu^N_i$
is given by the map $\tau_N$ such that  $\tau_N(x)=b_j$ for all $x\in x_j+Q_N$,
thus $\lim_{N\to\infty} W(\mu^i_N,\mu_i)\leq \lim_{N\to\infty}\frac{1}{N}= 0$.
Moreover, a transference plan $T^1_N\in P(V[\mu_1],V[\mu^N_1])$ 
is given by the map $\phi_N:T\R^n\to T\R^n$ such that
$\phi_N(x,v)=(\tau_N(x), v(\tau_N(x)))$, so $\phi_N(x,v)=(b_j,v(b_j))$ 
for all $x\in x_j+Q_N$ and $v\in\R^n$.
We construct similarly a tranference plan $T^2_N\in P(V[\mu^2_N],V[\mu_2])$.
By definition, $\pi_{13}\# T^1_N\in P^{opt}(\mu_1,\mu^N_1)$
and $\pi_{13}\# T^2_N\in P^{opt}(\mu^2_N,\mu_2)$. Moreover:
\begin{equation}\label{eq:est-T1_N}
\int_{(T\R^n)^2} |v-w| \ dT^1_N(x,v,y,w) \leq
\sup_j \sup_{x\in x_j+Q_N}|v(x)-v(b_j)| =\frac{L(v,R)}{N},
\end{equation}
where $Supp(\mu_i)\subset B(0,R)$ and $L(v,R)$ is the Lipschitz constant
of $v$ on $B(0,R)$. The same estimate holds true for  for $T^2_N$.
By Proposition \ref{prop:ODE-H3}, (H3) holds true for $\mu^N_i$,
then for every $\epsilon>0$ and $N$
there exists $T_N\in P(V[\mu^1_N],V[\mu^2_N])$, with
$\pi_{13}\#T_N\in  P^{opt}(\mu^1_N,\mu^2_N)$, such that:
\begin{equation}\label{eq:est-T_N}
\int_{T\R^n\times T\R^n} |v-w| \ dT_N(x,v,y,w)\leq \WW(V[\mu^1_N],V[\mu^2_N])+\epsilon\leq L(v,R)\ W(\mu^1_N,\mu^2_N)+\epsilon.
\end{equation}
We can compose the transference plans $T^1_N$, $T_N$ and $T^2_N$, see
Lemma 5.3.2, Remark 5.3.3 and Section 7.1 of \cite{AGS}.
More precisely, there exists
$\tilde{T}_N $    such that 
$\tilde{\pi}_{12}\#\tilde{T}_N = T^1_N$,
$\tilde{\pi}_{23}\#\tilde{T}_N = T_N$,
$\tilde{\pi}_{34}\#\tilde{T}_N = T^2_N$,
and
$\tilde{\pi}_{14}\#\tilde{T}_N\in P(V[\mu_1],V[\mu_2])$, where
$\tilde{\pi}_{ij}$ is the projection on the $i$-th and $j$-th components of 
the Cartesian product $(T\R^n)^4$. Moreover, it holds:
\begin{eqnarray}\label{eq:est-vw-T}
&\int_{(T\R^n)^2} |v-w| \ d(\tilde{\pi}_{14}\#\tilde{T}_N)(x,v,y,w)\leq
&\int_{(T\R^n)^2} |v-w| \ d(T^1_N+T_N+T^2_N)(x,v,y,w)\leq\nonumber\\
&L(v,R)\ W(\mu^1_N,\mu^2_N)+\epsilon +\frac{2\,L(v,R)}{N},
\end{eqnarray}
where we used (\ref{eq:est-T1_N}) and (\ref{eq:est-T_N}), and also:
\begin{eqnarray}\label{eq:est-xy-T}
&\int_{(T\R^n)^2} |x-y| \ d(\tilde{\pi}_{14}\#\tilde{T}_N)(x,v,y,w)\leq
&\int_{(T\R^n)^2} |x-y| \ d(T^1_N+T_N+T^2_N)(x,v,y,w)\leq\nonumber\\
&W(\mu_1,\mu^N_1)+ W(\mu^N_1,\mu^N_2) +
W(\mu^N_2,\mu_2).
\end{eqnarray}
The sequence $\tilde{\pi}_{14}\#\tilde{T}_N$ is tight, thus narrowly relatively compact,
see Lemma 5.2.2 and Theorem 5.1.3 of \cite{AGS}. Therefore, there exists a subsequence 
narrowly converging to $\tilde{T}\in P(V[\mu_1],V[\mu_2])$.
The transport costs are lower semicontinuous for narrow convergence,
see Proposition 7.13 of \cite{AGS}, so:
\[
\int_{(T\R^n)^2} |v-w| \ d\tilde{T}(x,v,y,w)\leq \liminf_{N\to\infty}
\int_{(T\R^n)^2} |v-w| \ d(\tilde{\pi}_{14}\#\tilde{T}_N)(x,v,y,w),
\]
thus by (\ref{eq:est-vw-T}), we get:
\begin{equation}\label{eq:est-tildeT}
\int_{(T\R^n)^2} |v-w| \ d\tilde{T}(x,v,y,w)\leq 
L(v,R)\ W(\mu_1,\mu_2)+\epsilon.
\end{equation}
Moreover: 
\[
\int_{(T\R^n)^2} |x-y| \ d\tilde{T}(x,v,y,w)\leq
\liminf_{N\to\infty} \int_{(T\R^n)^2} |x-y| \ d(\tilde{\pi}_{14}\#\tilde{T}_N)(x,v,y,w) 
\leq
\]
using (\ref{eq:est-xy-T}):
\[
\liminf_{N\to\infty}\ (W(\mu_1,\mu^N_1)+ W(\mu^N_1,\mu^N_2) +
W(\mu^N_2,\mu_2)) = W(\mu_1,\mu_2),
\]
thus $\pi_{13}\#\tilde{T}\in  P^{opt}(\mu_1,\mu_2)$. Therefore:
$\WW(V[\mu_1],V[\mu_2])\leq \int_{(T\R^n)^2} |v-w| \ d\tilde{T}(x,v,y,w)$
and, by (\ref{eq:est-tildeT}) and the arbitrariety of $\epsilon$,
we conclude that (H3) holds true on $\PP_c(\R^n)$.\\
Now, consider a finite sum of Dirac deltas $\mu=\sum_i m_i\delta_{x_i}$
and indicate by $x_i(\cdot)$ the unique solution
to the Cauchy problem $\dot x=v(x)$, $x(0)=x_i$.
The sequence of LAS is given by $\mu^N(t)=\sum_i m_i\delta_{x^N_i(t)}$,
where $x^N_i(t)$ is an (Euler) approximate solution of $x_i(\cdot)$. 
Given $T>0$, $x^N_i(\cdot)$ converges uniformly on $[0,T]$ to $x_i(\cdot)$,
thus $\mu^N$ converges in Wasserstein to $\mu(t) =\sum_i m_i\delta_{x_i(t)}$
uniformly on $[0,T]$.
We conclude applying Theorem \ref{th:uniq-sem-LAS}.
\end{proof}
We also have the following:
\begin{prop}\label{prop:lin-transp}
Assume $v$ is locally Lipschitz continuous with sublinear growth, then
the solution to the Cauchy problem for $\dot \mu=V^v[\mu]$
with initial datum $\mu_0$
is the unique solution to the transport equation:
\[
\mu_t+div(v\ \mu)=0,\quad \mu(0)=\mu_0.
\]
\end{prop}
The proof follows immediately from uniqueness of weak solutions 
to the transport equation, see \cite{Villani}.

\subsection{A natural monoid structure for PVFs}
We now describe a natural monoid structure and scalar product on the set
of PVFs, build upon the connections between vector fields and PVFs.\\ 
First we define a fiber convolution for measures on $T\R^n$ with same
marginal on the base. More precisely, given
$\mu_b\otimes \mu_x$, $\mu_b\otimes \nu_x\in \PP(T\R^n)$,
for every $B\subset T\R^n$ we define
\[
(\mu_b\otimes \mu_x)\ast_f (\mu_b\otimes \nu_x)(B)=
\int_{R^n} \left(\int_{(R^n)^2} \chi_B(x,v+w)\ d\mu_x(v)d\nu_x(w)\right)\ d\mu_b(x).
\]
Given two PVFs $V_1$, $V_2$, we denote $V_i[\mu]=\mu\otimes \nu^i_x$ ($\nu^i_x=\nu^i_x[\mu]$) the disintegration of $V_i$
on base-fiber of $T\R^n$, then we can define: 
\[
(V^1\oplus_f V^2)(\mu)=
(\mu\otimes \nu^1_x)\ast_f (\mu\otimes \nu^2_x).
\]
We have the following:
\begin{prop}\label{prop:Vec-PVF}
The operation $\oplus_f$ defines an abelian monoid structure over the set of PVFs.
\end{prop}
\begin{proof}
Commutativity and associativity follows from the same property of convolution
of measures (and linearity of the integration over the base).
The neutral element is given by $V[\mu]=\mu\otimes \delta_0$.
\end{proof}
Notice that every PVF $V^v$ is invertible and its inverse is $V^{-v}$,
but other elements are not invertible, thus $\oplus_f$ does not define
a group structure. However, the sum of two vector fields $v_i$ 
is mapped to the fiber-convolution of their PVFs, indeed:
\[
V^{v_1+v_2} (\mu)=\mu\otimes \delta_{v_1(x)+v_2(x)}
= (\mu\otimes \delta_{v_1(x)}) \ast_f (\mu\otimes \delta_{v_2(x)}).
\]
For every $\lambda\in\R$ and $B\subset T\R^n$, set:
\[
(\lambda\cdot_f V_i)[\mu](B)=
\int_{R^n} \left(\int_{R^n} \chi_B(x,\lambda v)\ d\nu^i_x[\mu](v)\right)\ d\mu(x).
\]
Let us denote by $Vec(\R^n)$ the set of locally Lipschitz vector fields
with sublinear growth endowed with the usual vector space structure,
and by $PVec(\R^n)$ the set of PVFs satisfying (H1) and (H3) endowed
with the operations $\oplus_f$ and $\cdot_f$, then we have:
\begin{prop}
The map $v\to V^v$ is a monoid isomorphism from $Vec(\R^n)$ to $PVec(\R^n)$.
Moreover $V^{\lambda v}=\lambda \cdot_f V^v$.
\end{prop}

\section{Finite speed diffusion and concentration}\label{sec:Ex}
In this section, we show examples of MDEs which reproduce diffusion 
and concentration phenomena.
The former can be obtained using PVF which depend on global quantities
and satisfy condition (H3) (while we also show that diffusion can not be
obtained by constant PVFs.) In particular we are able to model diffusions 
with uniformly bounded speed. 
Concentration is achieved by PVFs violating (H3), but still guaranteeing
convergence of LASs to unique limits and existence of Lipschitz semigroup.

\subsection{Diffusion}
Let us start proving a simple fact:
\begin{prop}\label{prop:const-V}
If a PVF $V$ does not depend on $\mu$, i.e. $V[\mu]=\mu\otimes \bar{V}$
for some $\bar{V}\in\PP_c(R^n)$,
then the solution to (\ref{eq:MDE-Cauchy}) obtained as limit LAS
is given by constant translation at speed $\bar{v}=\int v\,d\bar{V}$, 
i.e. for every Borel set $A$ one has $\mu(t)(A)=\mu_0(A-t\bar{v})$.
\end{prop}
\begin{proof}
Let $(\Omega, \B, P)$ be
a probability space and for every $W\in\PP_c(\R^n)$ we indicate by
$r_W:\Omega\to\R^n$ a random variable with distribution $W$. 
We also denote by $a^v_N:[-N,N]^n\to[-N,N]^n\subset\R^n$ 
the map corresponding to the projection used in the definition of $\A^v_N$ (see (\ref{eq:disc-v})). 
More precisely, 
we have $a^v_N(v_1,\ldots,v_n)=(\lfloor N\,v_1 \rfloor, \ldots, \lfloor N\,v_n\rfloor)$
($\lfloor s\rfloor = \sup \{n\in\N: n\leq s\}$.)\\
Let us first consider $\mu_0=\delta_0$. Then $\mu^N_\ell$ is the distribution
of the random variable:
\[
\Delta_N\sum_{i=1}^{\ell}a^v_N(r_{V[\mu^N_{\ell-1}]})=
\frac{1}{N}\sum_{i=1}^{\ell}a^v_N(r_{\bar{V}})=
\frac{\ell}{N}a^v_N(r_{\bar{V}}).
\]
Given $t\geq 0$,  set $\ell_N(t)=\lfloor Nt\rfloor$, then we have:
\[
r_{\mu^{N}_{\ell_N(t)}}-t\,\bar{v}=
\frac{\ell_N(t)}{N}a^v_N(r_{\bar{V}})-t\,\bar{v}=
\]
\[
\frac{\ell_N(t)}{N}(a^v_N(r_{\bar{V}})-r_{\bar{V}})+
\frac{\ell_N(t)}{N}(r_{\bar{V}}-\bar{v})+\left(\frac{\ell_N(t)}{N}-t\right)\bar{v}\doteq
I_1+I_2+I_3.
\]
Now $I_1$ tends to zero uniformly. Since $r_{\bar{V}}$ is bounded
and $E[r_{\bar{V}}]= \int v\,d\bar{V}=\bar{v}$,
by the Law of Large Numbers (see Theorem 10.12 in \cite{Folland}),
$(r_{\bar{V}}-\bar{v}) = \frac{1}{N}\sum_{i=1}^N (r_{\bar{V}}-\bar{v})\to 0$ almost surely
as $N\to\infty$.
Finally $I_3$ tends to zero uniformly by definition of $\ell_N(t)$.
The same proof works for finite sum of Dirac deltas, thus, by density,
on the whole set $\PP_c(\R^n)$.
\end{proof}

We now provide a first example of finite speed diffusion obtained by
letting the PVF $V$ depend on global properties of $\mu$.
\begin{example}\label{ex:1}
For every $\mu\in\PP_c(\R)$ define:
\[
B(\mu)=\sup \left\{x:\mu(]-\infty,x])\leq \frac12\right\}.
\]
Notice that we have $\mu(]-\infty,B(\mu)[)\leq \frac12\leq \mu(]-\infty,B(\mu)])$,
then we set $\eta=\mu(]-\infty,B(\mu)]) - \frac12$ so 
$\mu(\{B(\mu)\})=\eta+\frac12-\mu(]-\infty,B(\mu)[)$.
We define $V[\mu]=\mu\otimes \nu_x$, with:
\begin{equation}\label{eq:PVF-Bar}
\nu_x=\left\{
\begin{array}{ll}
\delta_{-1} & \textrm{if}\ x<B(\mu)\\
\delta_{1} & \textrm{if}\ x>B(\mu)\\
\eta\delta_{1}+
\left(\frac12-\mu(]-\infty,B(\mu)[)\right)\delta_{-1} & \textrm{if}\ x=B(\mu)
\end{array}
\right.
\end{equation}
\end{example}
We have the following:
\begin{prop}\label{prop:Ex1}
The PVF $V$ defined in (\ref{eq:PVF-Bar}) satisfies (H1) and (H3) and LASs admit a unique limit,
thus the conclusions of Theorem \ref{th:uniq-sem-LAS} holds true. Moreover,
the solution to (\ref{eq:MDE-Cauchy}), obtained as limit LASs $\mu^N$, satisfies:
\[
\mu(t)(A) = \mu_0((A\cap ]-\infty,B(\mu)-t[)+t)
+ \mu_0((A\cap ]B(\mu)+t,+\infty[)-t)
\]
\[
+ \eta \delta_{B(\mu)+t}(A)+ (\frac12-\mu(]-\infty,B(\mu)[))\delta_{B(\mu)-t}(A).
\]
In particular:
\begin{itemize}
\item[i)] The solution to (\ref{eq:MDE-Cauchy}) with $\mu_0=\delta_{x_0}$
is given by $\mu(t)=\frac12 \delta_{x_0+t}+\frac12 \delta_{x_0-t}$;
\item[ii)] The solution to (\ref{eq:MDE-Cauchy}) with $\mu_0=\chi_{[a,b]}\,\lambda$
(where $\chi$ is the indicator function and $\lambda$ is the Lebesgue measure) 
is given by $\mu(t)=\chi_{[a-t,\frac{a+b}{2}-t]}\lambda +
\chi_{[\frac{a+b}{2}+t,b+t]}\lambda$.
\end{itemize}
\end{prop}
\begin{proof}
The PVF $V$ satisfies (H1) by definition. Given two measures 
$\mu$, $\nu\in\PP_c(\R)$
notice that any optimal plan between $\mu$ and $\nu$
moves the mass of $\mu$ to the left, respectively right, of $B(\mu)$
to the mass of $\nu$ the left, respectively right, of $B(\nu)$
(see Theorem 2.18 and Remark 2.19 (ii) in \cite{Villani}.)
Therefore $\WW(V[\mu],V[\nu])=0$ and (H3) is trivially satisfied.\\
The other claims follow by direct computations.
\end{proof}

Example \ref{ex:1} can be generalized as follows.
\begin{example}
Consider an increasing map $\varphi:[0,1]\to\R$ and define
$$V_\varphi[\mu]=\mu\otimes \varphi\#( \chi_{[F_\mu(x-),F_\mu(x)]}\lambda)$$
where 
$F_\mu(x)=\mu(]-\infty,x])$, the cumulative distribution of $\mu$,
and $\lambda$ the Lebesgue measure. In simple words $V[\mu]$
moves the ordered masses with speed prescribed by $\varphi$.\\
Following the same proof of Proposition \ref{prop:Ex1}, we have that
$V_\varphi$ satisfies (H1) and (H3) if $\varphi$ is bounded.
If $\varphi$ is a diffeomorphism, the conclusions of Theorem \ref{th:uniq-sem-LAS} holds true
and the solution from $\delta_0$
is given by $g(t,x)\lambda$ with 
$$g(t,x)=\frac{1}{t\varphi '(\varphi^{-1}(\frac{x}{t}))}=
\frac{(\varphi^{-1})'(\frac{x}{t})}{t}.$$
For example if $\varphi(\alpha)=\alpha-\frac12$ then 
$g(t,x)=\frac{1}{t}\chi_{[-\frac{t}{2},\frac{t}{2}]}$
so we get uniformly distributed masses with maximal speed $1$.
For $\varphi(x)=4\, \rm{sgn}(\alpha-\frac12)\,(\alpha-\frac12)^2$
we get $g(t,x)=\frac{1}{4\sqrt{tx}}$ which is unbounded at $0$.
In general $V_\varphi$ gives rise to any $g$, which is solution to the equation
$g_t+\frac{x}{t}g_x=0$.
\end{example}

\begin{example}\label{ex:non-uni}
Let us go back to the question of uniqueness of solutions. Consider
the PVF $V_1$ defined in Example \ref{ex:1} and let 
$V_2[\mu]=\mu\otimes (\frac{1}{2}\delta_1+\frac{1}{2}\delta_{-1})$.
From Proposition \ref{prop:Ex1}, the solution to 
$\dot{\mu}=V_1[\mu]$, $\mu(0)=\delta_0$ is given by 
$\mu_1(t)=\frac{1}{2}\delta_t+\frac{1}{2}\delta_{-t}$.
From Proposition \ref{prop:const-V},
the solution to $\dot{\mu}=V_2[\mu]$, $\mu(0)=\delta_0$ is given by 
$\mu_2(t)=\delta_0$.
It is easy to check that $V_1[\delta_0]=V_2[\delta_0]$
and that $\mu_2$ satisfies (\ref{eq:sol-MDE}) both for $V_1$ and $V_2$.\\
It is also interesting to notice that the LASs $\mu^N$ for $\mu_1$
coincide with $\mu_1$ for every $N$.
On the other side, given $f\in\C^{\infty}_c$, 
$\int f\,d\mu_2(t)\equiv f(0)$ so $\frac{d}{dt}\int f\,d\mu_2(t)\equiv 0$.
Thus $\mu_1$ is trivially approximated by LASs, while
$\mu_2$ gives the trivial solution to (\ref{eq:sol-MDE}).
\end{example}

\subsection{Concentration}
It is well known that, to achieve existence and uniqueness of solutions 
to an ODE $\dot{x}=v(x)$, the locally Lipschitz condition on the vector field $v$ 
can be relaxed to a one-sided locally Lipschitz condition:
\begin{equation}\label{eq:one-sided-L}
\langle v(x)-v(y),x-y\rangle \leq L\ |x-y|^2,
\end{equation}
where $\langle \cdot, \cdot \rangle$ indicates the scalar product of $\R^n$.
Similarly we can relax condition (H3) as follows. Define:
\begin{eqnarray}\label{eq:WW'}
& \WW' (V_1,V_2)=\nonumber\\
& \inf \left\{
\int_{T\R^n\times T\R^n} \frac{\langle v-w, x-y\rangle }{|x-y|} 
\ dT(x,v,y,w)\ :\ T\in P(V_1,V_2), \ \pi_{13}\# T\in P^{opt}(\mu_1,\mu_2) \right\},
\end{eqnarray}
then we assume:
\begin{itemize}
\item[(H4)] For every $R>0$ there exists $K=K(R)>0$ such that if 
$Supp(\mu),Supp(\nu)\subset B(0,R)$ then
\begin{equation}
\WW'(V[\mu],V[\nu])\leq K\ W(\mu,\nu).
\end{equation}
\end{itemize}
We have the following:
\begin{theorem}\label{th:MDE'}
Given $V$ satisfying (H1) and (H4), and $T>0$, passing to the limit in LASs
we can define a Lipschitz semigroup of solutions to (\ref{eq:MDE-Cauchy}).
\end{theorem}
\begin{proof}
We can follow the same proof as for Theorem \ref{th:MDE} with the following
modification. To estimate the Wasserstein distance between $\mu^N_{\ell}$
and $\nu^N_{\ell}$, we proceed as follows. First notice that
for $a,b\in\R^n$ and $\epsilon>0$ we have
$\left. \frac{d}{d\epsilon}|a+\epsilon b|\right|_{\epsilon=0}
=\frac{\langle b, a \rangle}{|a|}$
thus  $|a+\epsilon b|=|a|+\epsilon \frac{\langle b, a\rangle}{|a|}+o(\epsilon)$.
We can then write:
\[
|(x+y)+\Delta_N(v+w)|\leq |x-y|+\Delta_N
\langle v-w, \frac{x-y}{|x-y|}| \rangle +o(\Delta_N).
\]
Now assumption (H4) guarantees that  (\ref{eq:rec-est12}) is still true 
and we can conclude in the same way as for Theorem  \ref{th:MDE}.
\end{proof}
Examples of contraction are obatined easily as follows.
\begin{example}
Consider an ODE $\dot{x}=v(x)$ with $v$ satisfying (\ref{eq:one-sided-L}) 
(with $L$ bounded on compact sets). Then condition (H4) holds true and we can
apply Theorem \ref{th:MDE'}.\\
A typical example is $v(x)=-{\rm sgn} (x)\sqrt{|x|}$. If we start with a uniform mass
distributed on the interval $[-1,1]$ the solution converges in time $t=1$
to $\delta_0$.
\end{example}
\begin{example}
Consider a scalar conservation law:
\[
u_t+\nabla\cdot(a(t,x)\ u)=0,
\]
with $a$ satisfying $\langle a(t,x)-a(t,y),x-y\rangle \leq L\ |x-y|^2$
uniformly in $t$ and on compact sets. Then the conclusions of Theorem \ref{th:ODE} hold true
for the ODE $\dot{x}=a(t,x)$.\\
One can thus recover the results of \cite{PoupaudRascle} for 
the compressive case (e.g. $a(t,x)=-\rm{sgn}(x)$.)
\end{example}

\section{Mean-field limits for multi-particle systems}\label{sec:kin}
A typical example of multi-particle system is given by the system of ODEs:
\begin{equation}\label{eq:kin-sys}
\dot{x}_i=\frac{1}{m}\sum_{j=1}^m \phi(x_j-x_i),
\end{equation}
where $x_i\in\R^n$, $i=1,\ldots,m$, 
and $\phi$ is locally Lipschitz continuous and uniformly bounded. 
For every $m$ and $x(\cdot)=(x_1(\cdot),\ldots,x_m(\cdot))\in\R^{nm}$, 
solution to (\ref{eq:kin-sys}),
consider the empirical probability measure of $m$ particles:
\begin{equation}\label{eq:emp-meas}
\mu_m(t)=\frac{1}{m}\sum_{i=1}^m \delta_{x_i(t)}.
\end{equation}
A typical problem is to understand the limit of $\mu_m$ as $m\to\infty$
(see for instance \cite{Golse}) and applications include problems from
biology, crowd dynamics and other fields, see \cite{CFRT,CPT}.
Dobrushin (see \cite{Dobrushin}) proved convergence, for the 
Wasserstein metric topology, of the empirical probability measures to solutions
of the mean field equation:
\[
\mu_t+\nabla_x\cdot \left(\left(\int \phi(x-y)d\mu(y)\right)\mu\right)=0.
\]
Let us consider a more general model:
\begin{equation}\label{eq:gen-kin}
\dot{x}_i=v^m_i(x)
\end{equation}
where $x=(x_1,\ldots,x_m)$, $x_i\in\R^n$ and $v^m_i$ is 
locally Lipschitz continuous and uniformly bounded. 
We assume a condition of indistinguibility of particles.
For every $m$, we indicate by $\Sigma_m$ the set of permuations
$\sigma$ over the set $\{1,\ldots,m\}$.
Given $\sigma\in\Sigma_m$ and $x=(x_1,\ldots,x_m)$, $x_i\in\R^n$,
we define $x_\sigma=(x_{\sigma(1)},\ldots,x_{\sigma(m)})$.
Then we assume:
\begin{itemize}
\item[(IP)] For every $x=(x_1,\ldots,x_m)$, $x_i\in\R^n$,
and $\sigma\in\Sigma_m$, it holds $v^m_{\sigma(i)}(x)=v^m_i(x_\sigma)$.
\end{itemize}
Notice that, given two empirical measures $\mu_j=\sum_{i=1}^m \delta_{x^j_i}$,
$j=1,2$, the Wasserstein distance between them is given by:
\[
W(\mu_1,\mu_2)=\inf_{\sigma\in\Sigma_m}
\frac{1}{m}\sum_i |x^1_i-x^2_{\sigma(i)}|,
\]
then, setting $x^j=(x^j_1,\ldots,x^j_m)$, we can estimate:
\[
\frac{1}{m}\sum_i |v^m_i(x^1)-v^m_{\sigma(i)}(x^2)|=
\frac{1}{m}\sum_i |v^m_i(x^1)-v^m_i(x^2_{\sigma})|\leq
\]
\[
\frac{1}{m}\sum_i L^m_i\ |x^1-x^2_\sigma|=
\frac{1}{m}\left(\sum_i L^m_i\right) \left(\sum_k |x^1_k-x^2_{\sigma(k)}|^2\right)^{\frac12}\leq
\left(\sum_i L^m_i\right)\cdot \frac{1}{m} \sum_k |x^1_k-x^2_{\sigma(k)}|,
\]
where $L^m_i$ is a (local) Lipschitz constant of $v^m_i$.
Let $\Sigma_m(\mu_1,\mu_2)$ be the set of $\sigma\in\Sigma_m$
which realize the Wasserstein distance $W(\mu_1,\mu_2)$, then:
\begin{eqnarray}\label{eq:est-kin}
\inf_{\sigma\in\Sigma_m(\mu_1,\mu_2)} \frac{1}{m}\sum_i |v^m_i(x^1)-v^m_{\sigma(i)}(x^2)| &\leq&
\inf_{\sigma\in\Sigma_m(\mu_1,\mu_2)}
\left(\sum_i L^m_i\right)\cdot \frac{1}{m} \sum_k |x^1_k-x^2_{\sigma(k)}|
\nonumber\\
&=& \left(\sum_i L^m_i\right)\cdot W(\mu_1,\mu_2).
\end{eqnarray}
The left-hand side of (\ref{eq:est-kin})
is precisely the term appearing in the definition of $\WW$ (see (\ref{eq:WW}))
if $V$ is a PVF corresponding to the system (\ref{eq:gen-kin}).
Assume there exist uniform bounds on the Lipschitz constants 
of $v^m_i$ and a PVF obtained as limit as $m\to\infty$ in the following sense:
\begin{itemize}
\item[(A)] If $L^m_i(R)$ is the Lipschitz constant of $v^m_i$ over the set $B(0,R)$,
then $\sup_m \sum_i L^m_i(R)<+\infty$ for every $R$.\\
Moreover, there exists a PVF $V$ such that the following holds true.
For every sequence $x^N=(x^N_1,\ldots,x^N_{m(N)})$, $N\in\N$, $m(N)\in\N$,
$x^N_i\in\R^n$, define $\mu_N=\frac{1}{m(N)}\sum_{i=1}^{m(N)}\delta_{x^N_i}$.
If there exists $R>0$ such that $|x^N_i|\leq R$ and $\mu\in\PP_c(\R^n)$ such that
$\lim_{N\to\infty} W(\mu_N,\mu)=0$, then
\begin{equation}\label{eq:(A)}
\lim_{N\to\infty}
\WW\left(\frac{1}{m(N)}\sum_{i=1}^{m(N)}\delta_{(x^N_i,v^{m(N)}_i(x^N))},V[\mu]\right)=0.
\end{equation}
\end{itemize}
Then we have the following:
\begin{theorem}\label{th:MDE-kinetic}
Consider the system (\ref{eq:gen-kin}), assume $v^m_i$ locally Lipschitz and bounded, (IP) and
(A) hold true, and denote by $V$ the PVF given by (A). 
Then $V$ satisfies (H3) and
the empirical probability measures (\ref{eq:emp-meas}),
where $x=(x_1,\ldots,x_m)(t)$ solves (\ref{eq:gen-kin}),
are solutions to the MDE $\dot{\mu}=V[\mu]$.
Moreover, there exists a unique Lipschitz semigroup for the MDE
whose trajectories coincide with the empirical probability measures
for finite sums of Dirac deltas.
\end{theorem}
\begin{proof}
Let $x=(x_1,\ldots,x_m)$, $x_i\in\R^n$,  and $\mu=\frac{1}{m}\sum_{i=1}^{m}\delta_{x_i}$, 
then by taking the constant sequence in (A) $\mu_N\equiv \mu$, $m(N)=m$, we deduce
$\WW(\frac{1}{m}\sum_{i=1}^{m}\delta_{(x_i,v^{m}_i(x))},V[\mu])=0$,
thus, by Lemma \ref{lem:WT}, $V[\mu]=\frac{1}{m}\sum_{i=1}^{m}\delta_{(x_i,v^{m}_i(x))}$.
Now, given $\mu\in\PP_c(\R^n)$, let $\mu_N$ be a sequence of finite sums of Dirac deltas as in (A)
with $W(\mu_N,\mu)\to 0$. From (\ref{eq:(A)}) and Lemma \ref{lem:WT}, 
we deduce $W^{T\R^n}(V[\mu],V[\mu_{N+1}])\to 0$.
Therefore we can uniquely define $V$ on the whole $\PP_c(\R^n)$ by approximation.\\
Property (H1) for $V$ follows from the boundedness of $v^m_i$.
To prove (H3), consider $\mu_i\in\PP_c(\R^n)$, $i=1,2$ and let
$\mu^i_N$ be sequences as in (A) such that $\lim_{N\to\infty} W(\mu^i_N,\mu_i)=0$.
For every $N$ and $\epsilon_N>0$,  
there exists $T_N\in P(V[\mu^1_N],V[\mu^2_N])$, with
$\pi_{13}\#T_N\in  P^{opt}(\mu^1_N,\mu^2_N)$, 
$T^1_N\in P(V[\mu_1],V[\mu^1_N])$, 
with $\pi_{13}\#T^1_N\in  P^{opt}(\mu_1,\mu^1_N)$,
and  $T^2_N\in P(V[\mu^2_N],V[\mu_2])$, with
$\pi_{13}\#T^2_N\in  P^{opt}(\mu^2_N,\mu_2)$,
 such that:
\[
\int_{T\R^n\times T\R^n} |v-w| \ dT_N(x,v,y,w)\leq \WW(V[\mu^1_N],V[\mu^2_N])+\epsilon_N,
\]
\[
\int_{T\R^n\times T\R^n} |v-w| \ dT^i_N(x,v,y,w)\leq \WW(V[\mu^i_N],V[\mu_i])+\epsilon_N, \qquad i=1,2.
\]
We can compose the transference plans $T^1_N$, $T_N$ and $T^2_N$, see 
Lemma 5.3.2, remark 5.3.3 and Section 7.1 of \cite{AGS}, thus there exists
$\tilde{T}_N $ such that 
$\tilde{\pi}_{12}\#\tilde{T}_N = T^1_N$,
$\tilde{\pi}_{23}\#\tilde{T}_N = T_N$,
$\tilde{\pi}_{34}\#\tilde{T}_N = T^2_N$,
and
$\tilde{\pi}_{14}\#\tilde{T}_N\in P(V[\mu_1],V[\mu_2])$, where
$\tilde{\pi}_{ij}$ is the projection on the $i$-th and $j$-th components of 
the Cartesian product $(T\R^n)^4$. Moreover, we have:
\[
\int_{(T\R^n)^2} |v-w| \ d(\tilde{\pi}_{14}\#\tilde{T}_N)(x,v,y,w)\leq
\int_{(T\R^n)^2} |v-w| \ d(T^1_N+T_N+T^2_N)(x,v,y,w)=
\]
\[
\WW(V[\mu_1],V[\mu^1_N])+\WW(V[\mu^1_N],V[\mu^2_N])+
\WW(V[\mu^2_N],V[\mu_2])+3\epsilon_N,
\]
and
\[
\int_{(T\R^n)^2} |x-y| \ d(\tilde{\pi}_{14}\#\tilde{T}_N)(x,v,y,w)\leq
\int_{(T\R^n)^2} |x-y| \ d(T^1_N+T_N+T^2_N)(x,v,y,w)=
\]
\[
W(\mu_1,\mu^1_N) +W(\mu^1_N,\mu^2_N) +
W(\mu^2_N,\mu_2).
\]
The sequence $\tilde{\pi}_{14}\#\tilde{T}_N$ is tight thus narrowly relatively compact
(Lemma 5.2.2 and Theorem 5.1.3 of \cite{AGS}), 
then there exists a subsequence converging to $\tilde{T}\in\PP((T\R^n)^2)$. 
The transportation costs are narrowly lower semicontinuous 
(Proposition 7.13 of \cite{AGS}), thus
we have that $\tilde{T}\in P(V[\mu_1],V[\mu_2])$ and:
\[
\int_{(T\R^n)^2} |v-w| \ d\tilde{T}(x,v,y,w)\leq \liminf_{N\to\infty}
\int_{(T\R^n)^2} |v-w| \ d(\tilde{\pi}_{14}\#\tilde{T}_N)(x,v,y,w).
\]
Moreover: 
\[
\int_{(T\R^n)^2} |x-y| \ d\tilde{T}(x,v,y,w)\leq
\liminf_{N\to\infty} \int_{(T\R^n)^2} |x-y| \ d(\tilde{\pi}_{14}\#\tilde{T}_N)(x,v,y,w) = W(\mu_1,\mu_2),
\]
thus $\pi_{13}\#\tilde{T}\in  P^{opt}(\mu_1,\mu_2)$.
Then:
\[
\WW(V[\mu_1],V[\mu_2])\leq
\int_{(T\R^n)^2} |v-w| \ d\tilde{T}(x,v,y,w)\leq
\]
\[
\leq \liminf_{N\to\infty}
\int_{(T\R^n)^2} |v-w| \ d(\tilde{\pi}_{14}\#\tilde{T_N})(x,v,y,w)\leq
\]
\[
\liminf_{N\to\infty}
\left(\WW(V[\mu_1],V[\mu^1_N])+\WW(V[\mu^1_N],V[\mu^2_N])+
\WW(V[\mu^2_N],V[\mu_2])+3\epsilon_N\right).
\]
Now we can choose $\epsilon_N\to 0$ as $N\to\infty$
and by (\ref{eq:(A)}) the first and third addendum in parenthesis tend to zero.
By (\ref{eq:est-kin}), the second addendum can be bounded by
$\sup_m \sum_i L^m_i(R)\ W(\mu_1,\mu_2)$, thus it follows:
\[
\WW(V[\mu_1],V[\mu_2])\leq \left(\sup_m \sum_i L^m_i(R)\right)\
W(\mu_1,\mu_2),
\]
then, by (A), $V$ satisfies (H3).\\
From Theorem \ref{th:MDE} there exists a Lipschitz semigroup
of solutions to the MDE $\dot{\mu}=V[\mu]$, obtained as limit of LASs. 
Moreover, using the local Lipschitz continuity of $v^m_i$, we can define 
a Dirac germ coinciding with the empirical probability measures
(which, in turn, coincide with the unique limit of LASs.)
Then we conclude applying Theorem \ref{th:uniq-sem}.
\end{proof}
We easily obtain the following:
\begin{corollary}\label{cor:kin}
Consider (\ref{eq:kin-sys}) with $\phi$ bounded and locally Lipschitz.
Then the conclusions of Theorem \ref{th:MDE-kinetic} hold true.
\end{corollary}
\begin{proof}
The system (\ref{eq:kin-sys}) can be written as (\ref{eq:gen-kin})
with $v^m_i (x) = \frac{1}{m}\sum_{j=1}^m \phi(x_j-x_i)$.
The uniform boundedness of $v^i_m$ follows from the boundedness of $\phi$.
Moreover, if $L_\phi(R)$ is the Lipschitz constant of $\phi$ on $B(0,R)$, then
$L^m_i(R)=\frac{1}{m} L_\phi(R)$ and $\sup_m \sum_i L^m_i(R)=L_\phi(R)$.
Finally, defining $V[\mu]=\mu\otimes \int_{\R^n} \phi(x-y)d\mu(y)$,
(A) follows from the local Lipschitz continuity of $\phi$.
We conclude applying Theorem \ref{th:MDE-kinetic}.
\end{proof}

\begin{remark}
Kinetic models with concentration phenomena were studied in a number
of papers, see for instance \cite{Bertozzi}. These models are not expected
to verify condition (H4), however they exhibit uniqueness of forward trajectories
for empirical measures. It would be natural to apply the MDE theory to prolong
solutions past blow-up times.
\end{remark}

\section*{Appendix}
{\textsc{Proof of of Lemma} \ref{lem:distr-sol}.}
Assume first that $\mu$ satisfies (\ref{eq:MDE-distr}) 
then for $g(t,x)=a(t)f(x)$, with $a\in{\C}^\infty_c(]0,T[)$ and $g\in{\C}^\infty_c(\R^n)$
we have:
\[
0=a(T) \int_{R^n} f(x)\,d\mu(T)(x) - a(0) \int_{R^n} f(x)\,d\mu_0(x) \ =
\]
\[
\int_0^T   a'(s) \int_{R^n} f(x)\ d\mu(s)\,\,ds+
\int_0^T a(s) \int_{TR^n}  (\nabla_x f(x)\cdot v)\ dV[\mu(s)](x,v)\ ds.
\]
Since $a$ is arbitrary we obtain that (\ref{eq:sol-MDE}) holds true
for every $f$.\\
Conversely the same computation show that if $\mu$ is a solution to to (\ref{eq:MDE-Cauchy})
then (\ref{eq:MDE-distr}) is satisfied for every $f(t,x)=a(t)g(x)$ thus
we conclude by density of such functions in ${\C}^\infty_c([0,T]\times \R^n)$.
{$\hfill\Box$\vspace{0.1 cm}\\}

{\textsc{Proof of of Lemma} \ref{lem:unif-comp}.}
The sequence $\mu_N$ is tight, i.e. for every $\epsilon>0$
there exists a compact set $K_\epsilon\subset \R^n$ such that for all $N$ it holds
$\mu_N(\R^n\setminus K_\epsilon)\leq\epsilon$. 
This is trivially satisfied taking $K_\epsilon=B(0,R)$.
Then, by Prokhorov Theorem (see Theorem 5.1.3 of \cite{AGS}) 
there exists a subsequence converging narrowly to $\mu\in\PP_c(\R^n)$,
i.e. $\int f\,d\mu_N\to \int f\,d\mu$ for every $f:\R^n\to\R$ continuous and bounded.
Since the moments $\int |x|d\mu_N$ are uniformly bounded, 
$W(\mu_N(t),\mu(t))\to 0$ (see Proposition 7.1.5 of \cite{AGS}).
{$\hfill\Box$\vspace{0.1 cm}\\}

{\textsc{Proof of of Lemma} \ref{lem:Sem-est}.}
The function:
\[
\liminf_{h\to 0+} \frac{1}{h} W(S_h \mu(s),\mu(s+h))
\]
is measurable and bounded. Measurability follows from observing
that the incremental ratios are continuous for fixed $h$ and taking
the infimum over $h\in\Q$, while boundedness from Lipschitz
continuity of the semigroup trajectories and of $\mu(\cdot)$.\\
Define:
\[
\psi(s)=W(S_{t-s}\mu(s),S_t\mu(0)),\quad
x(s)=\psi(s)-e^{Kt}\int_0^s \liminf_{h\to 0+} \frac{1}{h} W(S_h \mu(r),\mu(r+h))\,dr.
\]
Notice that $\psi(0)=x(0)=0$ and $\psi$ and $x$ are Lipschitz continuous. Therefore
for Rademacher Theorem $\dot\psi(s)$ and $\dot x(s)$ are defined for almost every $s$.
Moreover, by Lebesgue Theorem, $\phi$ is approximately continuous for almost every $s$.
Therefore, for almost every $s$, we have:
\[
\dot{x}(s)=\dot{\psi}(s)-e^{Kt} \liminf_{h\to 0+} \frac{1}{h} W(S_h \mu(s),\mu(s+h)).
\]
Moreover:
\[
\psi(s+h)-\psi(s)=W(S_{t-(s+h)}\mu(s+h),S_t\mu(0))-W(S_{t-s}\mu(s),S_t\mu(0))\leq
\]
\[
W(S_{t-(s+h)}\mu(s+h),S_{t-s}\mu(s))=
W(S_{t-(s+h)}\mu(s+h),S_{t-(s+h)}S_h\mu(s))\leq e^{Kt} W(\mu(s+h),S_h\mu(s)),
\]
which implies 
\[
\dot{\psi}(s)\leq e^{Kt} \liminf_{h\to 0+} \frac{1}{h} W(\mu(s+h),S_h\mu(s)).
\]
thus $\dot{x}(s)\leq 0$ for almost every $s$. Finally $x(t)\leq 0$
which proves the Lemma.
{$\hfill\Box$\vspace{0.1 cm}\\}


\end{document}